\documentclass[11pt,reqno]{amsart}
\newtheorem{theorem}{Theorem}[section]
\newtheorem{proposition}[theorem]{Proposition}
\newtheorem{corollary}[theorem]{Corollary}
\newtheorem{remark}[theorem]{Remark}
\newtheorem{lemma}[theorem]{Lemma}

\newtheorem{definition}[theorem]{Definition}

\usepackage{appendix}
\usepackage{amssymb}
\usepackage{graphicx}
\usepackage{amsmath}
\usepackage{bbm}
\usepackage[all]{xy}
\usepackage[latin1]{inputenc} 
\usepackage{graphicx} 
\numberwithin{equation}{section}
\usepackage{amsmath}
\usepackage{hyperref}
\allowdisplaybreaks[4]
\begin{document}
\title{commutators of pre Lie $n$-algebras and $PL_{\infty}$-algebras}

\author{Mengjun Wang}
\address{Department of Mathematics, Nanjing University,
Nanjing, 210008, P.R.China} \email{wangmj@nju.edu.cn}

\author{Zhixiang Wu}
\address{School of  Mathematical  Sciences, Zhejiang University,
Hangzhou, 310027, P.R.China} \email{wzx@zju.edu.cn}

\thanks{The authors are sponsored
by  NNSFC (No.11871421), ZJNSF (No. LY17A010015)}
 \subjclass[2000]{Primary 17D25, 18G55, Secondary 55U35}

\keywords{$n$-ary algebras, pre Lie algebras, left-symmetric algebras, $L_{\infty}$-algebras, $A_{\infty}$-algebras, $PL_{\infty}$-algebras,   homotopy algebras, commutators, pre Lie $n$-algebras}

\begin{abstract}
We show that a $PL_{\infty}$-algebra $V$ can be described by a nilpotent coderivation of degree $-1$ on coalgebra $P^*V$. Based on this result, we can generalise the result of T. Lada and show that every $A_{\infty}$-algebra carries a $PL_{\infty}$-algebra structure and every $PL_{\infty}$-algebra carries an $L_{\infty}$-algebra structure. In particular, we obtain a pre Lie $n$-algebra structure on an arbitrary partially associative  $n$-algebra and deduce pre Lie $n$-algebras are $n$-Lie admissible.
 \end{abstract}

\setcounter{tocdepth}{1}
\maketitle
\section{Introduction}\label{intro}
Left-symmetric algebras were introduced by A. Cayley \cite{c} in 1896 as a kind of rooted tree algebras and bacame being noticed after Vinberg \cite{v} in 1960 and Koszul \cite{ko} in 1961 introduced them in the study of convex homogeneous
cones and affinely manifolds. Recall that a left-symmetric algebra is a space $V$ endowed by a bilinear map $\mu:V\otimes V\to V$  satisfying   $$(x,y,z)=(y,x,z)$$ for all $x,y,z\in V$, where $(x,y,z):=\mu(\mu(x,y),z)-\mu(x,\mu(y,z))$. The opposite algebras of left-symmetric algebras are called right-symmetric algebras and they are both called pre Lie algebras. It is easy to see that every associative algebra is a pre Lie algebra. Any pre Lie algebra $(V,\mu)$ is a Lie-admissible algebra, i.e. the commutator $[x,y]:=\mu(x,y)-\mu(y,x)$ defines a Lie bracket on $V$. 

Many generalizations of pre Lie algebras has been widely studied as well. Homotopy pre Lie algebras ($PL_{\infty}$-algebras), for instant, were developed in \cite{cl} in the context of operad, and the concept of generalized pre-Lie algebras of order $n$ was introduced in \cite{pbg} without specific expression formulae for $n>3$. Similar generalizations of associative algebras and Lie algebras were introduced in \cite{j1,j,laj,go}. The purpose of this paper is to analyse the relation of these $n$-ary and homotopy algebra structures of associative, pre Lie, Lie type. Inspired by T. Lada \cite{la}, we first show that a $PL_{\infty}$-algebra structure on $V$ is equivalent to a nilpotent coderivation of degree $-1$ on coalgebra $P^*V$. By coalgebra maps between corresponding coalgebras of $A_{\infty},PL_{\infty}$ and $L_{\infty}$-algebras, we can obtain a $PL_{\infty}$-algebra structure on an $L_{\infty}$-algebra and a $L_{\infty}$-algebra structure on a $PL_{\infty}$-algebra. As a special case, we can finally give the commutators of $n$-ary algebras. The main results can be summarised as follows:
\begin{enumerate}
	\item [$\bullet$] Theorem \ref{tpl} which states that the $PL_{\infty}$-algebra structure on $V$ can be extended as a nilpotent coderivation of degree $-1$ on coalgebra $P^*V$.
	\item [$\bullet$] Theorem \ref{tc} which gives the relation of homotopy algebras and Corollary \ref{co} which gives the relation of $n$-ary algebras. 
\end{enumerate}

The paper is organised as follows. In Section \ref{s1}, we  provide some preliminaries and introduce a simple way to define the algebra expression formulae of pre Lie type. There are two different definitions of homotopy algebras in the type of associative, pre Lie and Lie, and we refer to them as degree $-1$ version and degree $n-2$ version. We show that $n$-ary algebras in the three types can be identified with homotopy algebras of degree $n-2$ version in the same type, while homotopy algebras of degree $-1$ version are closely related to coalgebras presented in Section \ref{s2}. We illustrate these two versions of homotopy algebras in same type are equivalent and homotopy algebras can be characterized by coderivations of corresponding coalgebras in Section \ref{s2}.
 
 With the three coalgebras and coalgebra maps between them in Section \ref{s2}, we derive the relation among   homotopy algebras in Section \ref{s3} by their equivalent characterizations in Section \ref{s2}. Since an $n$-ary algebra can be identified with a special   homotopy algebra by Section \ref{s1}, we get the corresponding relation among $n$-ary algebras.

\section{Preliminaries}\label{s1}
In this paper, we work over a field $\mathbb{K}$ of characteristic 0 and all the vector spaces are over $\mathbb{K}$. The symmetric group of the set $\{1,2,\cdots,n\}$ is denoted by  $\mathbb{S}_n$. While  $Sh(i_1,\cdots,i_m)$ is the subset of  $\mathbb{S}_n$ consisting of all $ (i_1,\cdots,i_m) $-unshuffles of $\mathbb{S}_n$, where $i_1+\cdots+i_m=n$.  Recall that an $ (i_1,\cdots,i_m) $-unshuffle is an element in $\mathbb{S}_n$   such that  $$\sigma(1+\sum\limits_{t=0}^{k-1}i_t)<\cdots<\sigma(\sum\limits_{t=0}^{k}i_t),\mbox{ for all }k=1,2,\cdots,m.$$ It is well-kown that $\sum\limits_{\sigma\in\mathbb{S}_n}\sigma$ is a nonzero integral in Hopf algebra $\mathbb{K}\mathbb{S}_n$. We always use $w_n$ to denote the integral of $\mathbb{K}\mathbb{S}_n$ in the sequel.

For any vector spaces $V$ and $W$ over the field $\mathbb{K}$, we use $Hom(V,W)$ to denote the space of all $\mathbb{K}$-linear maps from $V$ to $W$. The notation  $V\otimes W$ means $V\otimes_{\mathbb{K}}W$, the tensor product of $V$ and $W$ over the field $\mathbb{K}$. We use $\otimes ^nV$ to denote the space $\underbrace{V\otimes V\otimes\cdots\otimes V}_n$. It is well-know that $\otimes^nV$ is a right $\mathbb{K}\mathbb{S}_n$-module with the following action
$$\rho_{\sigma_1}(x_1\otimes\cdots\otimes x_n)=sgn(\sigma_1)(x_{\sigma_1(1)}\otimes\cdots\otimes x_{\sigma_1(n)})$$
for $\sigma_1\in\mathbb{S}_n$ and $x_1\otimes x_2\otimes\cdots\otimes x_n\in\otimes^nV$. The invariant subspace of $\otimes^nV$ under this action is denoted by $\wedge^nV$.
The identity endomorphism of $V$ is denoted by $id_V$ and $id_{\otimes^nV}$ is simply denoted by $I_n$.

Further assume that $V$ is  a $ \mathbb{Z} $-graded vector space $V:=\oplus_{n\in\mathbb{Z}}V^n$.  We follow \cite{s} for the terminology on the category of graded vector spaces. 
For any  $x\in V^n$ for some $n\in\mathbb{Z}$, we say that  $x$ is of homogeneous with degree $n$. The degree of a homogeneous element $x$ is denoted by  $|x|$.  If $x_i\in V$ are homogeneous, then the  degree of either $x_1\otimes\cdots\otimes x_n \in \otimes^nV$ or $x_1\wedge\cdots\wedge x_n\in\wedge^n V$ is defined as  $\sum\limits_{i=1}^{n}|x_i|$. 
 Let $f:V\to W$ be a map of graded vector spaces. Then $f$ is called a homogeneous linear map of cohomological degree $ n $ if $f(V^i)\subseteq W^{i+n}$ for any $n\in\mathbb{Z}$. The cohomogical degree  of a homogeneous linear map $f$ is denoted by $|f|.$
Suppose that $f:V\to V^{\prime}$ and $g:W\to W^{\prime}$  are two homogeneous linear maps. Then  the tensor product of $f$ and $g$, denoted by $f\otimes g$,   is a homogeneous linear map
from $V\otimes W$ to $V'\otimes W'$ determined by $$(f\otimes g)(x\otimes y):=(-1)^{n|g|}f(x)\otimes g(y)$$ for any $x\in V^n, y \in W$.

 For any transposition $(i,i+1)\in\mathbb{S}_n$ and $x_1\wedge x_2\wedge \cdots\wedge x_n\in\wedge^n V$, we have 
 $$x_1\wedge x_2\wedge \cdots\wedge x_n=(-1)^{|x_i||x_{i+1}|}x_1\wedge\cdots\wedge x_{i+1}\wedge x_i\wedge\cdots\wedge x_n.$$ Replace the transposition by an arbitrary element $\sigma $ in $\mathbb{S}_n$, and we can get the Koszul sign $\epsilon(\sigma):=\epsilon(\sigma;x_1,\cdots,x_n)$ \cite{m} recursively by a transposition decomposition of $\sigma$. Specifically, $$x_1\wedge x_2\wedge\cdots\wedge x_n=\epsilon(\sigma;x_1,\cdots,x_n)x_{\sigma(1)}\wedge x_{\sigma(2)}\wedge \cdots\wedge x_{\sigma(n)}.$$ 
 We simplified $\epsilon(\sigma;x_1,\cdots,x_n)$ as $\epsilon(\sigma)$ sometimes.

\begin{remark}\label{rm1}
	By definition, $x_{\tau(1)}\wedge\cdots\wedge x_{\tau(n)}=\epsilon(\sigma;x_{\tau(1)},\cdots,x_{\tau(n)})x_{\tau\sigma(1)}\wedge\cdots\wedge x_{\tau\sigma(n)}$ for $\sigma,\tau\in\mathbb{S}_n$. Converting both sides of the equation to multiples of $x_1\wedge\cdots\wedge x_n$, we have $$\epsilon(\tau;x_1,\cdots,x_n)=\epsilon(\sigma;x_{\tau(1)},\cdots,x_{\tau(n)})\epsilon(\tau\sigma;x_1,\cdots,x_n).$$ Since the value of $\epsilon$ is $\pm1$, the above equation can be expressed as $$\epsilon(\sigma;x_{\tau(1)},\cdots,x_{\tau(n)})=\epsilon(\tau\sigma;x_1,\cdots,x_n)\epsilon(\tau;x_1,\cdots,x_n):=\epsilon(\tau\sigma)\epsilon(\tau).$$
	\end{remark}

Similar to the case when $V$ is non-graded, $\otimes^nV$ is a right $\mathbb{K}\mathbb{S}_n$-module, where the action is given by $$\rho^{(1)}_{\sigma}(x_1\otimes\cdots\otimes x_n):=\epsilon(\sigma)(x_{\sigma(1)}\otimes\cdots\otimes x_{\sigma(n)}).$$ With this action, we can prove that
 $\wedge^nV$ is the space of coinvariants $(\otimes^nV)_{\mathbb{S}_n}:=(\otimes^nV)/(\rho^{(1)}_{\sigma}(\textbf{x})-\textbf{x},\sigma\in\mathbb{S}_n,\textbf{x}\in\otimes^nV)$ \cite{lv}. If a linear mapping $\hat{\mu}_n$  from $\otimes ^nV$ to $V$ satisfies $ \hat{\mu}_n=\hat{\mu}_n\circ (\rho_{\sigma}^{(1)}\otimes I_1) $ for $\sigma\in \mathbb{S}_{n-1}$, then it can be regarded as  a linear map from $\wedge^{n-1}V\otimes V$ to $V$. Similarly, any mapping $\hat{\mu}$ from $\otimes^nV$ to $V$ satisfying  $\hat{\mu}_n=\hat{\mu}_n\circ \rho_{\sigma}^{(1)}$ for any $\sigma\in\mathbb{S}_n$ can be viewed as  a linear map from  $\wedge^nV$ to $V$. 

With the previous preparation, we can  recall definitions of  $A_{\infty}$-algebras in \cite{j},  $L_{\infty}$-algebras in \cite{laj}, and $PL_{\infty}$-algebras in  \cite{cl} as follows.

\begin{definition}\label{dapl}
Let $V$ be a graded vector space equipped with a collection $\{\hat{\mu}_n:\otimes^n V\to V,n\geq1\}$ of homogeneous linear maps of cohomological degree $-1$. Then $(V,\{\hat{\mu}_n\})$ is 
		\begin{enumerate}
			\item [$\bullet$] an $A_{\infty}$-algebra if\begin{equation}\label{m2}
			\sum\limits_{i+j=n+1}\sum\limits_{m=0}^{i-1}\hat{\mu}_i\circ(I_{m}\otimes\hat{\mu}_j\otimes I_{i-m-1})=0, \forall n\geq 1,
			\end{equation}
			\item [$\bullet$]\label{pl2} a $PL_{\infty}$-algebra if \begin{equation}
			\begin{cases}\label{pl2}
			\hat{\mu}_n=\hat{\mu}_n\circ (\rho_{\sigma}^{(1)}\otimes I_1),\mbox{ for }\sigma\in \mathbb{S}_{n-1},\\
			\sum\limits_{i+j=n+1}\sum\limits_{m=0}^{i-1}\frac{1}{(i-1)!(j-1)!}\hat{\mu}_i\circ(I_{m}\otimes\hat{\mu}_j\otimes I_{i-m-1})\circ(\rho^{(1)}_{w_{i+j-2}}\otimes I_1)=0,\forall n\geq 1,
			\end{cases}
			\end{equation}
			\item [$\bullet$]\label{l2} an $L_{\infty}$-algebra if \begin{equation}
			\begin{cases}
			\hat{\mu}_n=\hat{\mu}_n\circ \rho_{\sigma}^{(1)},\mbox{ for }\sigma\in \mathbb{S}_n,\\
			\sum\limits_{i+j=n+1}\sum\limits_{m=0}^{i-1}\frac{1}{(i-1)!j!}\hat{\mu}_i\circ(I_{m}\otimes\hat{\mu}_j\otimes I_{i-m-1})\circ\rho^{(1)}_{w_{i+j-1}}=0, \forall n\geq 1.
			\end{cases}
			\end{equation} 
			\end{enumerate}
	
\end{definition}
\begin{remark}\label{rapl} (1) The equation \ref{pl2} can be replaced by 
		 $$\begin{array}{l}
		\sum\limits_{i+j=n+1}(\sum\limits_{\sigma \in Sh(j-1,1,i-2)}\epsilon(\sigma)\hat{\mu}_i(\hat{\mu}_j(x_{\sigma(1)},\cdots,x_{\sigma(j)}),x_{\sigma(j+1)},\cdots,\\ x_{\sigma{(i+j-2)}},x_{i+j-1})=\sum\limits_{\sigma \in Sh(i-1,j-1)}(-1)^{1+(\sum\limits_{r=1}^{i-1}|x_{\sigma(r)}|)}\epsilon(\sigma)\hat{\mu}_i(x_{\sigma(1)},\cdots,x_{\sigma(i-1)},\\
		\hat{\mu}_j(x_{\sigma(i)},\cdots,x_{\sigma(i+j-2)},\cdots,x_{i+j-1}))).
	\end{array}$$
	
(2)  Replace ``$\rho^{(1)}_{w_{i+j-2}}\otimes I_1$'' by ``$ I_1\otimes\rho^{(1)}_{w_{i+j-2}} $'', and we can get the notion of right-symmetric $\infty$-algebras which is exactly that of $PL_{\infty}$-algebras in \cite{cl}.

(3) We notice that for $\{\hat{\mu}_n\}$ satisfying equations (\ref{pl2}), the opposite operations $\{\hat{\mu}^{op}_n\}$ defined by
		$$\hat{\mu}^{op}_n(x_1,x_2,\cdots,x_n):=\hat{\mu}_n(x_n,x_{n-1},\cdots,x_1)$$ are not $PL_{\infty}$-algebras in \cite{cl} in general. In fact, it is the tensor rules of maps that contribute to this phenomenon. We can demonstraste the procedure in the following example.
		\begin{align*}
		&(g\circ(I_1\otimes f\otimes I_1))(x\otimes y_1\otimes y_2\otimes z)=(-1)^{|f||x|}g(x\otimes f(y_1\otimes y_2)\otimes z),\\
		&(g^{op}\circ(I_1\otimes f^{op}\otimes I_1))(z\otimes y_2\otimes y_1\otimes x)=(-1)^{|f||z|}g(x\otimes f(y_1\otimes y_2)\otimes z).
		\end{align*} 
		That is also the reason why $PL_{\infty}$-algebras in \cite{cl} have different signs with equations (\ref{pl2}). Although there is such an obstruction for graded vector spaces, we can get corresponding notions of right-symmetric algebras simply by reversing left-symmetric operations in non-graded case.
\end{remark}

Another right $\mathbb{K}\mathbb{S}_n$-module action on $\otimes^nV$ is defined via $$\rho^{(2)}_{\sigma}(x_1\otimes\cdots\otimes x_n):=sgn(\sigma)\epsilon(\sigma)(x_{\sigma(1)}\otimes\cdots\otimes x_{\sigma(n)}).$$ Replacing $\rho^{(1)}$ by $ \rho^{(2)} $ and equipping the above structure equations with sign functions, we can give  an equivalent  definition of Definition  \ref{dapl}. Namely, 
 
\begin{definition}\label{dal}
Let $V$ be a graded vector space equipped with a collection $\{\mu_n:\otimes^n V\to V,n\geq1\}$ of homogeneous linear maps of cohomological degree $n-2$. $V$ is called 
		\begin{enumerate}
			\item [$\bullet$] an $A_{\infty}$-algebra if \begin{equation}\label{m1}
			\sum\limits_{i+j=n+1}\sum\limits_{m=0}^{i-1}(-1)^{j(i-m-1)+m}\mu_i\circ(I_{m}\otimes\mu_j\otimes I_{i-m-1})=0,\forall n\geq 1
			\end{equation}
			\item [$\bullet$] a $PL_{\infty}$-algebra if \begin{equation}\label{pl1}
			\begin{cases}
			\mu_n=\mu_n\circ (\rho_{\sigma}^{(2)}\otimes I_1),\mbox{ for }\sigma\in \mathbb{S}_{n-1},\\
			\sum\limits_{i+j=n+1}\sum\limits_{m=0}^{i-1}\frac{(-1)^{j(i-m-1)+m}}{(i-1)!(j-1)!}\mu_i\circ(I_{m}\otimes\mu_j\otimes I_{i-m-1})\circ(\rho^{(2)}_{w_{i+j-2}}\otimes I_1)=0,\forall n\geq 1,
			\end{cases}
			\end{equation}
			\item [$\bullet$] an $L_{\infty}$-algebra if \begin{equation}\label{l1}
			\begin{cases}
			\mu_n=\mu_n\circ \rho_{\sigma}^{(2)},\mbox{ for }\sigma\in \mathbb{S}_n,\\
			\sum\limits_{i+j=n+1}\sum\limits_{m=0}^{i-1}\frac{(-1)^{j(i-m-1)+m}}{(i-1)!j!}\mu_i\circ(I_{m}\otimes\mu_j\otimes I_{i-m-1})\circ\rho^{(2)}_{w_{i+j-1}}=0, \forall n\geq 1.
			\end{cases}
			\end{equation} 
			\end{enumerate}
\end{definition}

These two different forms of homotopy algebras in the same type are equivalent. The detailed  proof of this is presented in subsection \ref{s32}.
\begin{remark}
	Explicitly, for $n\geq1$ and $x_1,x_2,\cdots,x_{i+j-1}\in V$, equation (\ref*{pl1}) means that
	\begin{align*}
	&\sum\limits_{i+j=n+1}(\sum\limits_{\sigma \in Sh(j-1,1,i-2)}(-1)^{j(i-1)}sgn(\sigma)\epsilon(\sigma)\mu_i(\mu_j(x_{\sigma(1)},\cdots,x_{\sigma(j)}),x_{\sigma(j+1)},\cdots,x_{\sigma{(i+j-2)}},\\
	&x_{i+j-1})=\sum\limits_{\sigma \in Sh(i-1,j-1)}(-1)^{i+j(\sum\limits_{r=1}^{i-1}|x_{\sigma(r)}|)}sgn(\sigma)\epsilon(\sigma)\mu_i(x_{\sigma(1)},\cdots,x_{\sigma(i-1)},\mu_j(x_{\sigma(i)},\cdots,\\
	&x_{\sigma(i+j-2)},\cdots,x_{i+j-1}))).
	\end{align*}
\end{remark}

In Definition \ref{dal}, if  (\ref{m1}) is replaced by
\begin{equation}
\sum\limits_{i=0}^{n-1}\mu\circ(I_i\otimes\mu\otimes I_{n-1-i})=0,
\end{equation} then $V$  is called a partially associative $n$-algebra in \cite{go}.
 Imitating constructing method of $PL_{\infty}$-algebras and referring to the definition of partially associative $n$-algebras, we introduce pre Lie $n$-algebras as follows.

\begin{definition}\label{pl} Suppose $V$ is a vector space and $\mu\in Hom(\otimes ^nV,V)$. Then $(V,\mu)$ is a \textbf{pre Lie $n$-algebra} if $\mu$ satisfies 
		\begin{equation}\label{eq28}
		\begin{cases}
		\mu=\mu\circ(\rho_{\sigma}\otimes I_1),\mbox{ for }\sigma\in \mathbb{S}_{n-1},\\
		\sum\limits_{i=0}^{n-1}\frac{(-1)^{i(n-1)}}{((n-1)!)^2}\mu\circ(I_i\otimes\mu\otimes I_{n-1-i})\circ(\rho_{w_{2n-2}}\otimes I_1)=0.
		\end{cases}
		\end{equation}
 \end{definition}
\begin{remark} (1)
In \cite{go}, a Lie $n$-algebra  is defined to be  a vector $V$ with $\mu\in Hom(\otimes ^nV,V)$ such that 
\begin{equation}
		\begin{cases}
		\mu=\mu\circ\rho_{\sigma},\mbox{ for }\sigma\in \mathbb{S}_n,\\
		\sum\limits_{i=0}^{n-1}\frac{(-1)^{i(n-1)}}{(n-1)!n!}\mu\circ(I_i\otimes\mu\otimes I_{n-1-i})\circ\rho_{w_{2n-1}}=0.
		\end{cases}
\end{equation}
Thus a Lie $n$-algebra in \cite{go}  can be regarded as a pre Lie $n$-algebra.

 (2) A pre Lie algebra in \cite{b}  is nothing but a pre Lie 2-algebra.
\end{remark}

Next, we will prove that a pre Lie $n$-algebra is exactly the left-symmetric version of generalized pre-Lie algebras of order $n$ in \cite{pbg}.
To achieve this aim, let us recall a result in \cite{w}.
For any vector space $V$, let  $C^n(V,V):=\{\mu\in Hom(\otimes^{n+1} V,V)|\mu=\mu\circ(\rho_{\sigma}\otimes I_1),\mbox{ for }\sigma\in \mathbb{S}_{n-1}\}$ and $C(V,V):=\oplus_{n\in\mathbb{N}}C^n(V,V)$. Then the following result holds. 

\begin{theorem}\cite{w} $C(V,V)$ is a graded Lie algebra with a  bracket given by 
	\begin{equation}
	[f,g]^{\circ}:=f\circ g-(-1)^{mn}g\circ f,\mbox{ for }f\in C^m(V,V),g\in C^n(V,V),
	\end{equation}
	where $f\circ g\in C^{m+n}(V,V)$ is defined by
	\begin{align}
	&(f\circ g)(x_1,\cdots x_{m+n+1})\notag\\
	=&\sum\limits_{\sigma \in Sh(n,1,m-1)}sgn(\sigma)f(g(x_{\sigma(1)},\cdots,x_{\sigma(n)},x_{\sigma(n+1)}),x_{\sigma(n+2)},\cdots,x_{\sigma(m+n)},x_{m+n+1})\notag\\
	&+(-1)^{mn}\sum\limits_{\sigma \in Sh(m,n)}sgn(\sigma)f(x_{\sigma(1)},\cdots,x_{\sigma(m)},g(x_{\sigma(m+1)},\cdots,x_{\sigma(m+n)},x_{m+n+1})).
	\end{align}
\end{theorem}
Then we get a necessary and sufficient condition of a pre Lie $n$-algebra.
\begin{lemma} Suppose that  $\mu\in C^{n-1}(V,V)$. Then  $(V,\mu)$ is  a pre Lie $n$-algebra if and only if $\mu \circ\mu=0$.
\end{lemma}
\begin{proof} Since 
	\begin{align*}
	&\sum\limits_{i=0}^{n-1}\frac{(-1)^{i(n-1)}}{((n-1)!)^2}\mu\circ(I_i\otimes\mu\otimes I_{n-1-i})\circ(\rho_{w_{2n-2}}\otimes I_1)(x_1,\cdots,x_{2n-1})\\
	=&\sum\limits_{\sigma \in \mathbb{S}_{2n-2}}sgn(\sigma)(\sum\limits_{i=0}^{n-2}\frac{(-1)^{i(n-1)}}{((n-1)!)^2}\mu(x_{\sigma(1)}, \cdots, x_{\sigma(i)}, \mu(x_{\sigma(i+1)}, \cdots, x_{\sigma(i+n)}), x_{\sigma(i+n+1)}, \cdots,\\
	&x_{\sigma(2n-2)}, x_{2n-1})+\frac{(-1)^{n-1}}{((n-1)!)^2}\mu(x_{\sigma(1)}, \cdots, x_{\sigma(n-1)}, \mu(x_{\sigma(n)}, \cdots, x_{\sigma(2n-2)}, x_{2n-1})))\\
	=&\sum\limits_{\sigma \in \mathbb{S}_{2n-2}}sgn(\sigma)(\sum\limits_{i=0}^{n-2}\frac{(-1)^{in}}{((n-1)!)^2}\mu(\mu(x_{\sigma(i+1)}, \cdots, x_{\sigma(i+n)}), x_{\sigma(1)}, \cdots, x_{\sigma(i)},  x_{\sigma(i+n+1)}, \cdots,\\
	& x_{\sigma(2n-2)}, x_{2n-1})+\frac{(-1)^{n-1}}{((n-1)!)^2}\mu(x_{\sigma(1)}, \cdots, x_{\sigma(n-1)}, \mu(x_{\sigma(n)}, \cdots, x_{\sigma(2n-2)}, x_{2n-1})))\\
	=&\frac{1}{n-1}(\sum\limits_{i=0}^{n-2}\sum\limits_{\sigma \in Sh(n-1,1,n-2)}sgn(\sigma)\mu(\mu(x_{\sigma(1)}, \cdots, x_{\sigma(n)}), x_{\sigma(n+1)}, \cdots, x_{\sigma(2n-2)},\\
	& x_{2n-1}))+(-1)^{n-1}\sum\limits_{\sigma \in Sh(n-1,n-1)}sgn(\sigma)\mu(x_{\sigma(1)}, \cdots, x_{\sigma(n-1)}, \mu(x_{\sigma(n)}, \cdots,\\
	& x_{\sigma(2n-2)}, x_{2n-1}))\\
	=&\sum\limits_{\sigma \in Sh(n-1,1,n-2)}sgn(\sigma)\mu(\mu(x_{\sigma(1)}, \cdots, x_{\sigma(n)}), x_{\sigma(n+1)}, \cdots, x_{\sigma(2n-2)}, x_{2n-1})+\\
	&(-1)^{n-1}\sum\limits_{\sigma \in Sh(n-1,n-1)}sgn(\sigma)\mu(x_{\sigma(1)}, \cdots, x_{\sigma(n-1)}, \mu(x_{\sigma(n)}, \cdots, x_{\sigma(2n-2)}, x_{2n-1}))\\
	=&(\mu\circ\mu)(x_1,\cdots,x_{2n-1}),
	\end{align*}for any $x_1,\cdots,x_{2n-1} \in V$, $\mu\circ \mu=0$ if and only (\ref{eq28}) holds.
\end{proof}

A short calculation reveals that if $(V,\mu)$ is a pre Lie algebra in Definition \ref{pl}, $(V,\mu^{op})$ is a generalized pre-Lie algebra of order $n$ in \cite{pbg}.
\begin{remark}
	Note that Lie $n$-algebras and $n$-Lie algebras are two different $n$-generalizations of Lie algebras \cite{rs} and $n$-Lie algebras are special Lie $n$-algebras \cite{go}. Correspondingly, our pre Lie $n$-algebras (generalized pre-Lie algebras of order $n$ in \cite{pbg}) are different from $n$-pre Lie algebras in \cite{pbg} and $n$-pre Lie algebras are special pre Lie $n$-algebras.
\end{remark}
In the remainder of this section, we will explain an $n$-ary algebra  (associative, Lie or pre Lie) is a special corresponding homotopy algebra in Definition \ref{dal}.

Although an $n$-ary algebra's structure equations are analogous to that of a homotopy algebra in Definition \ref{dapl}, there are differences 
 in the signs of structure equations when we take in elements. Since signs are determined by the parity of relevant numbers, P. Hanlon and M. Wachs get around this dilemma by superspaces, i.e. bigraded vector space in \cite{hw}. But it is not applicable to general spaces. It seems that we can solve this problem simply by regarding $V$ as a graded vector space concentrated in cohomological degree 0. In fact, such a graded vector space can only be equipped with non-zero bilinear map $\mu_2$ because $|\mu_n|=0$ if and only if $n=2$. So we take another tack.

For any $n$-ary algebra $(V,\mu)$, we construct an associated homotopy algebra $(\overline{V}=\oplus_{n}\bar{V}^i,\{\mu_i\})$,  where
\begin{eqnarray}
\overline{V}^i=\begin{cases}
V,&\mbox{if }i=0,n-2 \mbox{ or }2n-4,\\
0,&\mbox{otherwise },
\end{cases}
&\mbox{ and }&\mu_i=0 \mbox{ unless } i=n.
\end{eqnarray}

Thus for any non-zero homogeneous element $(x_1,\cdots,x_n)\in\otimes^n\overline{V}$, the corresponding degree $|x_1|+\cdots+|x_n|=k(n-2)$ for some non-negative integer $k$. With the forgetful image $(x'_1,\cdots,x'_n)$ in $\otimes^nV$  , we can give the function of $\mu_n$ as follow.
\begin{eqnarray}\label{mn}
\mu_n(x_1,\cdots,x_n)=\begin{cases}
\mu(x'_1,\cdots,x'_n),&\mbox{if }|x_1\otimes\cdots\otimes x_n|=0 \mbox{ or }n-2,\\
0,&\mbox{otherwise }.
\end{cases}
\end{eqnarray}
Note that the equals sign in equation (\ref{mn}) means the values of two functions are equal and $\mu_n:\otimes^n\overline{V}\to\overline{V}$ is a homogeneous linear map of cohomological degree $n-2$. Now we can identify an $n$-ary algebra with a homotopy algebra in Definition \ref{dal}.
\begin{proposition}\label{ni}
	\begin{enumerate}
		\item $(V,\mu)$ is a partially associative $n$-algebra if and only if $(\overline{V},\{\mu_i\})$ is an $A_{\infty}$-algebra.
		\item $(V,\mu)$ is a pre Lie $n$-algebra if and only if $(\overline{V},\{\mu_i\})$ is a $PL_{\infty}$-algebra.
		\item $(V,\mu)$ is a Lie $n$-algebra if and only if $(\overline{V},\{\mu_i\})$ is an $L_{\infty}$-algebra.
	\end{enumerate}
\end{proposition}

\section{Coalgebras, coderivations and   homotopy algebras}\label{s2}

In this section, we explain the equivalence of different forms of the same homotopy algebra in Section \ref{s1} and relate $PL_{\infty}$-algebras to coderivations of a coalgebra. Thus $V$ is always a graded vector space unless otherwise specified.

\subsection{Coalgebras and coalgebra maps between them}\label{s21}

Given a graded vector space $V$, there are three cofree objects on $V$ being of interest to us: the cofree coalgebra $T^*V$, the cofree commutative coalgebra $\wedge^*V$ and the cofree left Perm-coalgebra $P^*V$ in \cite{cl}. Next we present their graded structures and coalgebra structures.
\begin{enumerate}
	\item $T^*V:=\oplus_{n\geq1}(\otimes^nV)$ is a graded vector space equipped with a comultiplication map
	\begin{eqnarray}
	\Delta(x_1\otimes\cdots\otimes x_n)=\sum\limits_{i=1}^{n-1}(x_1\otimes\cdots\otimes x_i)\otimes(x_{i+1}\otimes\cdots\otimes x_n),
	\end{eqnarray}
	\item The comultiplication of $\wedge^*V:=\oplus_{n\geq1}(\wedge^nV)$ is given by
	\begin{eqnarray}
	\Delta(x_1\wedge\cdots\wedge x_n)=\sum\limits_{i=1}^{n-1}\sum\limits_{\sigma \in Sh(i,n-i)}\epsilon(\sigma)(x_{\sigma(1)}\wedge\cdots\wedge x_{\sigma(i)})\otimes(x_{\sigma(i+1)}\wedge\cdots\wedge x_{\sigma(n)}).
	\end{eqnarray}
	\item The $n$-part of $P^*V$ is denoted by $P^nV:=(\wedge^{n-1}V)\otimes V$. Its comultiplication is defined by
	\begin{align}
	&\Delta(x_1\wedge\cdots\wedge x_{n-1}\otimes x_n)\notag\\
	=&\sum\limits_{i=1}^{n-1}\sum\limits_{\sigma \in Sh(i-1,1,n-i-1)}\epsilon(\sigma)(x_{\sigma(1)}\wedge\cdots\wedge x_{\sigma(i-1)}\otimes x_{\sigma(i)})\otimes(x_{\sigma(i+1)}\wedge\cdots\wedge x_{\sigma(n-1)}\otimes x_n).
	\end{align}
\end{enumerate}

Expanding the coalgebra map in \cite{lam}, we get the following result.
\begin{lemma}\label{cd}
	There is a commutative diagram of coalgebras
	\begin{displaymath}
	\xymatrix{\wedge^*V\ar[rr]^{\hat{\alpha}}\ar[dr]^{\hat{\beta}}&&T^*V\\
	&P^*V\ar[ur]^{\hat{\gamma}}&}
	\end{displaymath} where $$\hat{\alpha}(x_1\wedge\cdots\wedge x_n):=\sum\limits_{\sigma \in \mathbb{S}_{n}}\epsilon(\sigma)x_{\sigma(1)}\otimes\cdots\otimes x_{\sigma(n)},$$
	$$\hat{\beta}(x_1\wedge\cdots\wedge x_n):=\sum\limits_{\sigma \in Sh(n-1,1)}\epsilon(\sigma)x_{\sigma(1)}\wedge\cdots\wedge x_{\sigma(n-1)}\otimes x_{\sigma(n)},$$ 
	$$\hat{\gamma}(x_1\wedge\cdots\wedge x_{n-1}\otimes x_n):=\sum\limits_{\sigma \in \mathbb{S}_{n-1}}\epsilon(\sigma)x_{\sigma(1)}\otimes\cdots\otimes x_{\sigma(n-1)}\otimes x_n$$ are injective coalgebra maps.
\end{lemma}
\begin{proof} Since  $w_n$ is an integral of $\mathbb{KS}_n$, $\hat{\alpha}=\sum\limits_{n\geq1}\rho_{w_{n}}^{(1)}$ and $\hat{\gamma}=\sum\limits_{n\geq1}\rho_{w_{n-1}}^{(1)}\otimes I_1$. Similarly, we can prove that  $\hat{\beta}=\sum\limits_{n\geq 1,\sigma \in Sh(n-1,1)}\rho_{\sigma}^{(1)}$.
Then we show that  $\hat{\gamma}\hat{\beta}=\hat{\alpha}$ according to the discussion in Remark \ref{rm1}. For any $x_1\wedge\cdots\wedge x_n\in \wedge^nV$,
	\begin{align*}
	&(\hat{\gamma}\hat{\beta})(x_1\wedge\cdots\wedge x_n)\\
	=&\sum\limits_{\sigma \in Sh(n-1,1)}\sum\limits_{\tau \in \mathbb{S}_{n-1}}\epsilon(\sigma)\epsilon(\tau;x_{\sigma(1)},\cdots,x_{\sigma(n-1)})x_{\tau\sigma(1)}\otimes\cdots\otimes x_{\tau\sigma(n-1)}\otimes x_{\sigma(n)}\\
	=&\sum\limits_{\sigma \in Sh(n-1,1)}\sum\limits_{\substack{\tau \in \mathbb{S}_{n}\\\tau(n)=n}}\epsilon(\sigma)\epsilon(\tau;x_{\sigma(1)},\cdots,x_{\sigma(n)})x_{\tau\sigma(1)}\otimes\cdots\otimes x_{\tau\sigma(n)}\\
	=&\sum\limits_{\sigma \in \mathbb{S}_n}\epsilon(\sigma)x_{\sigma(1)}\otimes\cdots\otimes x_{\sigma(n)}\\
	=&\hat{\alpha}(x_1\wedge\cdots\wedge x_n).
	\end{align*}
	
Next we show that  $\hat{\alpha},\hat{\beta},\hat{\gamma}$ are  injective coalgebra maps.  Take $\hat{\alpha}$ as an example.  We need to verify that $(\hat{\alpha}\otimes\hat{\alpha})\Delta=\Delta\hat{\alpha}$. Let $x_i\in V$ be homogeneous elements. Then
	\begin{align*}
	&((\hat{\alpha}\otimes\hat{\alpha})\Delta)(x_1\wedge\cdots\wedge x_n)\\
	=&\sum\limits_{i=0}^{n}\sum\limits_{\sigma \in Sh(i,n-i)}\epsilon(\sigma)\hat{\alpha}(x_{\sigma(1)}\wedge\cdots\wedge x_{\sigma(i)})\otimes\hat{\alpha}(x_{\sigma(i+1)}\wedge\cdots\wedge x_{\sigma(n)})\\
	=&\sum\limits_{i=0}^{n}\sum\limits_{\sigma \in \mathbb{S}_n}\epsilon(\sigma)(x_{\sigma(1)}\otimes\cdots\otimes x_{\sigma(i)})\otimes(x_{\sigma(i+1)}\otimes\cdots\otimes x_{\sigma(n)})\\
	=&(\Delta\hat{\alpha})(x_1\wedge\cdots\wedge x_n).
	\end{align*}
Denote the canonical epimorphism $\pi_n:\otimes^nV\to\wedge^nV,x_1\otimes\cdots\otimes x_n\mapsto x_1\wedge\cdots\wedge x_n$, and $\pi:=\sum\limits_{n\geq1}\frac{1}{n!}\pi_n$. A short calculation reveals that $\pi\hat{\alpha}=Id_{\wedge^*V}$, which means $\hat{\alpha}$ is injective.

Similarly, we can prove that $\hat{\beta},\hat{\gamma}$ are injective coalgebra maps.
\end{proof}

\subsection{Equivalent definitions of   homotopy algebras}\label{s32}
From Section 2, we know that there are two differential definitions of $A_{\infty}$ and $L_{\infty}$ algebras in \cite{j,k} and \cite{laj,lam}. Next, we will show that these two different definitions of these two kinds homotopy algebras are  equivalent. These homotopy algebras are characterized through coderivations of  some coalgebras. Similarly,  we can be obtained the same conclusion for $PL_{\infty}$-algebras. To avoid confusion, we use  pure letter with subscript like $\mathfrak{a}_n$ to mean a map of degree $n-2$, letter with a hat like $\hat{\mathfrak{b}}_n$ to mean a map of degree $-1$ and letter with a tilde like $\tilde{\mathfrak{c}}_n$  to mean a coderivation. Furthermore, maps denoted by $\mathfrak{a}_n$ and $\hat{\mathfrak{a}}_n$ are in one-to-one correspondence, and so are $\hat{\mathfrak{a}}_n$ and $\tilde{\mathfrak{a}}$. 

For a graded vector space $V$, its suspension is denoted by $sV$, i.e. $(sV)^i=V^{i-1}$. In \cite{laj}, T. Lada and J. Stasheff present a bijection between the families of maps $\mu_n:\otimes^nV\to V$ of degree $n-2$ and maps $\hat{\mu}_n:\otimes^n(sV)\to sV$ of degree $-1$. Namely,
\begin{align}
\hat{\mu}_n(sx_1,\cdots,sx_n)=\begin{cases}
(-1)^{\sum\limits_{i=1}^{n/2}|x_{2i-1}|}s\mu_n(x_1,\cdots,x_n),&\mbox{ if $n$ is even,}\\
-(-1)^{\sum\limits_{i=1}^{(n-1)/2}|x_{2i}|}s\mu_n(x_1,\cdots,x_n),&\mbox{ if $n$ is odd.}
\end{cases}
\end{align}
Recall that a linear map $f:C\to C$ is a coderivation of a coalgebra $C$ if $$\Delta_Cf=(f\otimes Id_C+Id_C\otimes f)\Delta_C,$$where $\Delta_C$ is the comultiplication of the coalgebra $C$. A collection of maps $\{\hat{\mu}_n:\otimes^nV\to V\}$ of degree $-1$ can be uniquely extended as a coderivation $\tilde{\mu}:T^*V\to T^*V$ where the component $\tilde{\mu}:\otimes^kV\to \otimes^lV$ is defined as $$\sum\limits_{i=0}^{l-1}I_i\otimes\hat{\mu}_{k-l+1}\otimes I_{l-i-1}.$$ Using this, one can prove that  three equivalent descriptions of $A_{\infty}$-algebras.
\begin{lemma}\label{tm}\cite{k}
	For a graded vector space $V$ the following statements are equivalent:
	\begin{enumerate}
		\item $(V,\{\mathfrak{m}_n\})$ satisfies equations (\ref{m1});
		\item $(sV,\{\hat{\mathfrak{m}}_n\})$ satisfies equations (\ref{m2});
		\item $\tilde{\mathfrak{m}}$ is a coderivation of coalgebra $T^*(sV)$ of degree $-1$ such that $\tilde{\mathfrak{m}}^2=0$.
	\end{enumerate}
\end{lemma}
Let $(V,\{\mathfrak{l}_n\})$ be an $L_{\infty}$-algebras in the sense of  Definition \ref{dal}. Then  $(sV,\{\hat{\mathfrak{l}}_n\})$ is proved to be an $L_{\infty}$-algebras in the sense of Definition \ref{dapl} (see \cite{laj}). As is discussed in Section 2, $\hat{\mathfrak{l}}_n$ can be seen as a linear map from $\wedge^n(sV)$ to $sV$. Note that a collection of maps $\{\hat{\mu}_n:\wedge^nV\to V\}$ of degree $-1$ is in one-to-one correspondence with a coderivation $\tilde{\mu}:\wedge^*V\to \wedge^*V$ where the component $\tilde{\mu}:\wedge^kV\to \wedge^lV$ is defined as $$\sum\limits_{i=0}^{l-1}\frac{1}{(l-1)!(k-l+1)!}(I_i\wedge\hat{\mu}_{k-l+1}\wedge I_{l-i-1})\circ\rho_{w_{k}}^{(1)}.$$ Then we can restate results in \cite{laj} in the following way. 
\begin{lemma}
	For a graded vector space $V$  the following statements are equivalent:
	\begin{enumerate}
		\item $(V,\{\mathfrak{l}_n\})$ satisfies equations (\ref{l1});
		\item $(sV,\{\hat{\mathfrak{l}}_n\})$ satisfies equations (\ref{l2});
		\item $\tilde{\mathfrak{l}}$ is a coderivation of coalgebra $\wedge^*(sV)$ of degree $-1$ such that $\tilde{\mathfrak{l}}^2=0$.
	\end{enumerate}
\end{lemma}
Since $PL_{\infty}$-algebras have equal status with $A_{\infty}$-algebras and $L_{\infty}$-algebras, we naturally consider similar equivalent characterizations of $PL_{\infty}$-algebras. For a $PL_{\infty}$-algebra $(V,\{\mathfrak{p}_n\})$ in Definition \ref{dal}, we can get an induced coderivation of coalgebra $P^*(sV)$ of degree $-1$. As a preliminary, we first show
\begin{lemma}\label{Per}For any $\sigma\in\mathbb{S}_{n-1}$, we have the following
 \begin{eqnarray}\label{per}
		(-1)^{\sum\limits_{i=1}^{n/2}|x_{\sigma(2i-1)}|}sgn(\sigma)\epsilon(\sigma;x_1,\cdots,x_{n-1})=(-1)^{\sum\limits_{i=1}^{n/2}|x_{2i-1}|}\epsilon(\sigma;sx_1,\cdots,sx_{n-1})
		\end{eqnarray} if $n$ is even,  and 
		\begin{eqnarray}\label{per1}
		(-1)^{\sum\limits_{i=1}^{(n-1)/2}|x_{\sigma(2i)}|}sgn(\sigma)\epsilon(\sigma;x_1,\cdots,x_{n-1})=(-1)^{\sum\limits_{i=1}^{(n-1)/2}|x_{2i}|}\epsilon(\sigma;sx_1,\cdots,sx_{n-1}), 
		\end{eqnarray}if $n$ is odd.
\end{lemma}
\begin{proof} In \cite{laj}, T. Lada and J. Stasheff illustrate this lemma by the example of $\sigma=(j,j+1)$ for some integer $j$. To facilitate readers, we present a complete proof here.
	
	 Assume that $n$ is even. Then $\sigma$ can be uniquely decomposed into a product of different transpositions which are in the form of $(k,k+1)$ by the following procedures.
	
	Suppose $\sigma$ is not the identity $(1)$ of $\mathbb{S}_{n-1}$. Let $j$ be the largest integer such that $\sigma(j)\neq j$. Then there is an integer $i$ smaller than $j$ such that $\sigma(i)=j$. Define $\sigma_1:=\sigma(i,i+1)(i+1,i+2)\cdots(j-1,j)$. If $\sigma_1=(1)$, we have $\sigma=(j-1,j)(j-2,j-1)\cdots(i,i+1)$. If $\sigma_1\neq(1)$, the largest integer $j_1$ such that $\sigma_1(j_1)\neq j_1$ is smaller than $j$. This process is repeated till we get the identity permutation. 
	
	We use $|\sigma|$ to refer to the number of transpositions in above discomposition of $\sigma$ and $|(1)|:=0$. Thus we can derive equation (\ref{per}) by induction on $|\sigma|$.
	
	The conclusion is obvious when $|\sigma|=0$. Suppose $\sigma$ is a non-identity permutation. If $|\sigma|=1$, $\sigma=(i,i+1)$ for some $i$.
	At this point we have 
	\begin{align*}
	\epsilon(\sigma;x_1,\cdots,x_{n-1})&=(-1)^{|x_i||x_{i+1}|},\\
	sgn(\sigma)&=-1,\\
	(-1)^{\sum\limits_{i=1}^{n/2}|x_{\sigma(2i-1)}|}&=(-1)^{(\sum\limits_{i=1}^{n/2}|x_{2i-1}|)+|x_i|+|x_{i+1}|}.
	\end{align*}
	Take this results into equation (\ref{per}), and we can easily finish the prove for $\sigma=(i,i+1)$. 
	\begin{align*}
	&(-1)^{\sum\limits_{i=1}^{n/2}|x_{\sigma(2i-1)}|}sgn(\sigma)\epsilon(\sigma;x_1,\cdots,x_{n-1})\\
	=&(-1)^{\sum\limits_{i=1}^{n/2}|x_{2i-1}|}(-1)^{1+|x_i|+|x_{i+1}|+|x_i||x_{i+1}|}\\
	=&(-1)^{\sum\limits_{i=1}^{n/2}|x_{2i-1}|}(-1)^{(1+|x_i|)(1+|x_{i+1}|)}\\
	=&(-1)^{\sum\limits_{i=1}^{n/2}|x_{2i-1}|}\epsilon(\sigma;sx_1,\cdots,sx_{n-1}).
	\end{align*}
	
	Now we assume the conclusion is true if $|\sigma|=k$ and consider the case that $|\sigma|=k+1$. With the decomposition of $\sigma$, we have $\sigma=\tau\delta$ where $|\tau|=k$ and $\delta=(i,i+1)$ for some $i$. Define $x_{\tau(1)}\otimes\cdots\otimes x_{\tau(n-1)}:=y_1\otimes\cdots\otimes y_{n-1}$. Then $x_{\sigma(1)}\otimes\cdots\otimes x_{\sigma(n-1)}:=y_{\delta(1)}\otimes\cdots\otimes y_{\delta(n-1)}$. Based on the properties of the Koszul sign, we verify equation (\ref{per}) as follows.
	\begin{align*}
	&(-1)^{\sum\limits_{i=1}^{n/2}|x_{\sigma(2i-1)}|}sgn(\sigma)\epsilon(\sigma;x_1,\cdots,x_{n-1})\\
	=&(-1)^{\sum\limits_{i=1}^{n/2}|y_{\delta(2i-1)}|}sgn(\tau)sgn(\delta)\epsilon(\delta;y_1,\cdots,y_{n-1})\epsilon(\tau;x_1,\cdots,x_{n-1})\\
	=&(-1)^{\sum\limits_{i=1}^{n/2}|y_{2i-1}|}\epsilon(\delta;sy_1,\cdots,sy_{n-1})sgn(\tau)\epsilon(\tau;x_1,\cdots,x_{n-1})\\
	=&\epsilon(\delta;sx_{\tau(1)},\cdots,sx_{\tau(n-1)})(-1)^{\sum\limits_{i=1}^{n/2}|x_{\tau(2i-1)}|}sgn(\tau)\epsilon(\tau;x_1,\cdots,x_{n-1})\\
	=&(-1)^{\sum\limits_{i=1}^{n/2}|x_{2i-1}|}\epsilon(\delta;sx_{\tau(1)},\cdots,sx_{\tau(n-1)})\epsilon(\tau;sx_1,\cdots,sx_{n-1})\\
	=&(-1)^{\sum\limits_{i=1}^{n/2}|x_{2i-1}|}\epsilon(\sigma;sx_1,\cdots,sx_{n-1}).
	\end{align*}
	Hence equation (\ref{per}) holds for any $\sigma\in\mathbb{S}_{n-1}$. Equation (\ref{per1}) can be calculated in an analogous way.
\end{proof}
Then we construct the coderivation associated to given $PL_{\infty}$-algebra structure maps.
\begin{proposition}\label{p}(1) 
Let $\{\mathfrak{p}_n:\otimes^nV\to V\}$ be a collection of linear maps of degree $n-2$. Then 
for $n\geq1$,  $\mathfrak{p}_n=\mathfrak{p}_n\circ (\rho_{\sigma}^{(2)}\otimes I_1)$ for any $\sigma\in\mathbb{S}_{n-1}$ if and only if $\hat{\mathfrak{p}}_n=\hat{\mathfrak{p}}_n\circ (\rho_{\sigma}^{(1)}\otimes I_1)$ for any $\sigma\in\mathbb{S}_{n-1}$.

(2) A collection of maps $\{\hat{\mathfrak{q}}_n:P^nV\to V\}$ of degree $-1$ can be uniquely extended as a coderivation $\tilde{\mathfrak{q}}:P^*V\to P^*V$ where the component $\tilde{\mathfrak{q}}:P^kV\to P^lV$ is defined as $$\frac{1}{(l-1)!(k-l)!}(\sum\limits_{i=0}^{l-2}I_i\wedge\hat{\mathfrak{q}}_{k-l+1}\wedge I_{l-i-2}\otimes I_1+I_{l-1}\wedge\hat{\mathfrak{q}}_{k-l+1})\circ(\rho_{w_{k-1}}^{(1)}\otimes I_1).$$ 
\end{proposition} 

\begin{proof} (1) Suppose $n$ is even.  For $\mathfrak{p}_n$ satisfying $\mathfrak{p}_n=\mathfrak{p}_n\circ (\rho_{\sigma}^{(2)}\otimes I_1)$, we have
		 \begin{align*}
		 &\hat{\mathfrak{p}}_n(sx_{\sigma(1)},\cdots,sx_{\sigma(n-1)},sx_n)\\
		 =&(-1)^{\sum\limits_{i=1}^{n/2}|x_{\sigma(2i-1)}|}s\mathfrak{p}_n(x_{\sigma(1)},\cdots,x_{\sigma(n-1)},x_n)\\
		 =&(-1)^{\sum\limits_{i=1}^{n/2}|x_{\sigma(2i-1)}|}sgn(\sigma)\epsilon(\sigma;x_1,\cdots,x_{n-1})s\mathfrak{p}_n(x_1,\cdots,x_n)\\
		 =&\epsilon(\sigma;sx_1,\cdots,sx_{n-1})(-1)^{\sum\limits_{i=1}^{n/2}|x_{2i-1}|}s\mathfrak{p}_n(x_1,\cdots,x_n)\\
		 =&\epsilon(\sigma;sx_1,\cdots,sx_{n-1})\hat{\mathfrak{p}}_n(sx_1,\cdots,sx_n).
		 \end{align*}
		  Conversely, for $ \hat{\mathfrak{p}}_n $ satisfying $\hat{\mathfrak{p}}_n=\hat{\mathfrak{p}}_n\circ (\rho_{\sigma}^{(1)}\otimes I_1)$, we have 
		 \begin{align*}
		 &\mathfrak{p}_n(x_{\sigma(1)},\cdots,x_{\sigma(n)})\\
		 =&(-1)^{\sum\limits_{i=1}^{n/2}|x_{\sigma(2i-1)}|}s^{-1}\hat{\mathfrak{p}}_n(sx_{\sigma(1)},\cdots,sx_{\sigma(n-1)},sx_n)\\
		 =&(-1)^{\sum\limits_{i=1}^{n/2}|x_{\sigma(2i-1)}|}\epsilon(\sigma;sx_1,\cdots,sx_{n-1})s^{-1}\hat{\mathfrak{p}}_n(sx_1,\cdots,sx_n)\\
		 =&(-1)^{\sum\limits_{i=1}^{n/2}|x_{2i-1}|}sgn(\sigma)\epsilon(\sigma;x_1,\cdots,x_{n-1})s^{-1}\hat{\mathfrak{p}}_n(sx_1,\cdots,sx_n)\\
		 =&sgn(\sigma)\epsilon(\sigma;x_1,\cdots,x_{n-1})\mathfrak{p}_n(x_1,\cdots,x_n).
		 \end{align*}
A similar discussion can be given in the case of $n$ is odd.

(2)  Since the component $\tilde{\mathfrak{q}}:P^nV\to V$ is exactly $\hat{\mathfrak{q}}_n$, the uniqueness is given. So all we need to do is to check $\tilde{\mathfrak{q}}$ is a coderivation. In fact, we only need to show their components from $P^kV$ to $P^iV\otimes P^jV$ are equal. Note that the component $\tilde{\mathfrak{q}}:P^kV\to P^lV$ is computed as follows.
		 \begin{align*}
		 &\tilde{\mathfrak{q}}(x_1\wedge\cdots\wedge x_{k-1}\otimes x_k)\\
		 =&\frac{1}{(l-1)!(k-l)!}(\sum\limits_{i=0}^{l-2}I_i\wedge\hat{\mathfrak{q}}_{k-l+1}\wedge I_{l-i-2}\otimes I_1+I_{l-1}\wedge\hat{\mathfrak{q}}_{k-l+1})\circ(\rho_{w_{k-1}}^{(1)}\otimes I_1)(x_1\wedge\cdots\wedge\\
		 & x_{k-1}\otimes x_k)\\
		 =&\sum\limits_{\sigma \in \mathbb{S}_{k-1}}\frac{\epsilon(\sigma)}{(l-1)!(k-l)!}(\sum\limits_{i=0}^{l-2}(-1)^{\sum\limits_{r=1}^i|x_{\sigma(r)}|}x_{\sigma(1)}\wedge\cdots\wedge x_{\sigma(i)}\wedge\hat{\mathfrak{q}}_{k-l+1}(x_{\sigma(i+1)},\cdots,\\
		 & x_{\sigma(i+k-l+1)})\wedge x_{\sigma(i+k-l+2)}\wedge\cdots\wedge x_{\sigma(k-1)}\otimes x_k\\
		 &+(-1)^{\sum\limits_{t=1}^{l-1}|x_{\sigma(t)}|}x_{\sigma(1)}\wedge\cdots\wedge x_{\sigma(l-1)}\wedge\hat{\mathfrak{q}}_{k-l+1}(x_{\sigma(l)},\cdots, x_{\sigma(k-1)}, x_k))\\
		 =&\sum\limits_{\sigma \in Sh(k-l,1,l-2)}\epsilon(\sigma)\hat{\mathfrak{q}}_{k-l+1}(x_{\sigma(1)},\cdots, x_{\sigma(k-l+1)})\wedge x_{\sigma(k-l+2)}\wedge\cdots\wedge x_{\sigma(k-1)}\otimes x_k\\
		 &+\sum\limits_{\sigma \in Sh(l-1,k-l)}(-1)^{\sum\limits_{t=1}^{l-1}|x_{\sigma(t)}|}\epsilon(\sigma)x_{\sigma(1)}\wedge\cdots\wedge x_{\sigma(l-1)}\otimes\hat{\mathfrak{q}}_{k-l+1}(x_{\sigma(l)},\cdots, x_{\sigma(k-1)}, x_k)
		 \end{align*}
		 Using this, we can obtain the component of $\Delta\tilde{\mathfrak{q}}:P^kV\to P^{i+j}V\to P^iV\otimes P^jV$.
		 \begin{align}\label{e1}
		 &\Delta\tilde{\mathfrak{q}}(x_1\wedge\cdots\wedge x_{k-1}\otimes x_k)\notag\\
		 =&\sum\limits_{\sigma \in Sh(k-i-j,1,i+j-2)}\epsilon(\sigma)\Delta(\hat{\mathfrak{q}}_{k-i-j+1}(x_{\sigma(1)},\cdots, x_{\sigma(k-i-j+1)})\wedge x_{\sigma(k-i-j+2)}\wedge\cdots\wedge x_{\sigma(k-1)}\otimes x_k)\notag\\
		 &+\sum\limits_{\sigma \in Sh(i+j-1,k-i-j)}(-1)^{\sum\limits_{t=1}^{i+j-1}|x_{\sigma(t)}|}\epsilon(\sigma)\Delta(x_{\sigma(1)}\wedge\cdots\wedge x_{\sigma(i+j-1)}\otimes\hat{\mathfrak{q}}_{k-i-j+1}(x_{\sigma(i+j)},\cdots,\notag\\
		 & x_{\sigma(k-1)}, x_k))\notag\\
		 =&\sum\limits_{\sigma \in Sh(k-i-j,1,i-2,1,j-1)}\epsilon(\sigma)(\hat{\mathfrak{q}}_{k-i-j+1}(x_{\sigma(1)},\cdots, x_{\sigma(k-i-j+1)})\wedge x_{\sigma(k-i-j+2)}\wedge\cdots\wedge x_{\sigma(k-j-1)}\notag\\
		 &\otimes x_{\sigma(k-j)})\otimes(x_{\sigma(k-j+1)}\wedge\cdots\wedge x_{\sigma(k-1)}\otimes x_k)\notag\\
		 &+\sum\limits_{\sigma \in Sh(i-1,k-i-j,1,j-1)}(-1)^{\sum\limits_{t=1}^{i-1}|x_{\sigma(t)}|}\epsilon(\sigma)(x_{\sigma(1)}\wedge\cdots\wedge x_{\sigma(i-1)}\otimes \hat{\mathfrak{q}}_{k-i-j+1}(x_{\sigma(i)},\cdots,  x_{\sigma(k-j)}))\notag\\
		 &\otimes(x_{\sigma(k-j+1)}\wedge\cdots\wedge x_{\sigma(k-1)}\otimes x_k)\notag\\
		 &+\sum\limits_{\sigma \in Sh(i-1,1,k-i-j,1,j-2)}(-1)^{\sum\limits_{t=1}^{i}|x_{\sigma(t)}|}\epsilon(\sigma)(x_{\sigma(1)}\wedge\cdots\wedge x_{\sigma(i-1)}\otimes x_{\sigma(i)})\otimes(\hat{\mathfrak{q}}_{k-i-j+1}(x_{\sigma(i+1)}\notag\\
		 &,\cdots, x_{\sigma(k-j+1)})\wedge x_{\sigma(k-j+2)}\wedge\cdots\wedge x_{\sigma(k-1)}\otimes x_k)\notag\\
		 &+\sum\limits_{\sigma \in Sh(i-1,1,j-1,k-i-j)}(-1)^{\sum\limits_{t=1}^{i+j-1}|x_{\sigma(t)}|}\epsilon(\sigma)(x_{\sigma(1)}\wedge\cdots\wedge x_{\sigma(i-1)}\otimes x_{\sigma(i)})\otimes(x_{\sigma(i+1)}\wedge\cdots\wedge\notag\\
		 &x_{\sigma(i+j-1)}\otimes\hat{\mathfrak{q}}_{k-i-j+1}(x_{\sigma(i+j)},\cdots, x_{\sigma(k-1)}, x_k))
		 \end{align}
		 
		 On the other hand, we consider the component $P^kV\to P^{k-j}V\otimes P^jV\to P^iV\otimes P^jV$ of $(\tilde{\mathfrak{q}}\otimes Id_{P^*V})\Delta$ and $P^kV\to P^iV\otimes P^{k-i}V\to P^iV\otimes P^jV$ of $(Id_{P^*V}\otimes\tilde{\mathfrak{q}})\Delta$.
		 \begin{align}\label{e2}
		 &((\tilde{\mathfrak{q}}\otimes Id_{P^*V})\Delta)(x_1\wedge\cdots\wedge x_{k-1}\otimes x_k)\notag\\
		 =&\sum\limits_{\sigma \in Sh(k-j-1,1,j-1)}\epsilon(\sigma)\tilde{\mathfrak{q}}(x_{\sigma(1)}\wedge\cdots\wedge x_{\sigma(k-j-1)}\otimes x_{\sigma(k-j)})\otimes(x_{\sigma(k-j+1)}\wedge\cdots\wedge x_{\sigma(k-1)}\otimes x_k)\notag\\
		   =&\sum\limits_{\sigma \in Sh(k-i-j,1,i-2,1,j-1)}\epsilon(\sigma)(\hat{\mathfrak{q}}_{k-i-j+1}(x_{\sigma(1)},\cdots, x_{\sigma(k-i-j+1)})\wedge x_{\sigma(k-i-j+2)}\wedge\cdots\wedge x_{\sigma(k-j-1)}\notag\\
		  &\otimes x_{\sigma(k-j)})\otimes(x_{\sigma(k-j+1)}\wedge\cdots\wedge x_{\sigma(k-1)}\otimes x_k)\notag\\
		  &+\sum\limits_{\sigma \in Sh(i-1,k-i-j,1,j-1)}(-1)^{\sum\limits_{t=1}^{i-1}|x_{\sigma(t)}|}\epsilon(\sigma)(x_{\sigma(1)}\wedge\cdots\wedge x_{\sigma(i-1)}\otimes \hat{\mathfrak{q}}_{k-i-j+1}(x_{\sigma(i)},\cdots,  x_{\sigma(k-j)}))\notag\\
		  &\otimes(x_{\sigma(k-j+1)}\wedge\cdots\wedge x_{\sigma(k-1)}\otimes x_k)
		 \end{align}
		 \begin{align}\label{e3}
		 &((Id_{P^*V}\otimes \tilde{\mathfrak{q}})\Delta)(x_1\wedge\cdots\wedge x_{k-1}\otimes x_k)\notag\\
		 =&\sum\limits_{\sigma \in Sh(i-1,1,k-i-1)}(-1)^{\sum\limits_{t=1}^{i}|x_{\sigma(t)}|}\epsilon(\sigma)(x_{\sigma(1)}\wedge\cdots\wedge x_{\sigma(i-1)}\otimes x_{\sigma(i)})\otimes\tilde{\mathfrak{q}}(x_{\sigma(i+1)}\wedge\cdots\wedge x_{\sigma(k-1)}\notag\\
		 &\otimes x_k)\notag\\
		 =&\sum\limits_{\sigma \in Sh(i-1,1,k-i-j,1,j-2)}(-1)^{\sum\limits_{t=1}^{i}|x_{\sigma(t)}|}\epsilon(\sigma)(x_{\sigma(1)}\wedge\cdots\wedge x_{\sigma(i-1)}\otimes x_{\sigma(i)})\otimes(\hat{\mathfrak{q}}_{k-i-j+1}(x_{\sigma(i+1)}\notag\\
		 &,\cdots, x_{\sigma(k-j+1)})\wedge x_{\sigma(k-j+2)}\wedge\cdots\wedge x_{\sigma(k-1)}\otimes x_k)\notag\\
		 &+\sum\limits_{\sigma \in Sh(i-1,1,j-1,k-i-j)}(-1)^{\sum\limits_{t=1}^{i+j-1}|x_{\sigma(t)}|}\epsilon(\sigma)(x_{\sigma(1)}\wedge\cdots\wedge x_{\sigma(i-1)}\otimes x_{\sigma(i)})\otimes(x_{\sigma(i+1)}\wedge\cdots\wedge\notag\\
		 &x_{\sigma(i+j-1)}\otimes\hat{\mathfrak{q}}_{k-i-j+1}(x_{\sigma(i+j)},\cdots, x_{\sigma(k-1)}, x_k))
		 \end{align}
		  
		  It can be easily seen that $(\ref{e1})=(\ref{e2})+(\ref{e3})$, which means that $\tilde{\mathfrak{q}}$ is a coderivation of $P^*V$.
\end{proof}

 Based on the one to one correspondence $\{\mathfrak{p}_n\}\leftrightarrow\{\hat{\mathfrak{p}}_n\}\leftrightarrow\tilde{\mathfrak{p}}$, we give the following theorem.
 
 \begin{theorem}\label{tpl}
	For a graded vector space $V$ the following statements are equivalent:
	\begin{enumerate}
		\item\label{ti1} $(V,\{\mathfrak{p}_n\})$ satisfies equations (\ref{pl1});
		\item\label{ti2} $(sV,\{\hat{\mathfrak{p}}_n\})$ satisfies equations (\ref{pl2});
		\item\label{ti3} $\tilde{\mathfrak{p}}$ is a coderivation of coalgebra $P^*(sV)$ of degree $-1$ such that $\tilde{\mathfrak{p}}^2=0$.
	\end{enumerate}
\end{theorem}

\subsection{Proof of Theorem \ref{tpl}}According to Proposition \ref{p}, the stability of $\mathfrak{p}_n$ under the action of $\rho^{(2)}$, the stability of $\hat{\mathfrak{p}}_n$ under the action of $\rho^{(1)}$ and the condition that $\tilde{\mathfrak{p}}$ is a coderivation are equivalent. So the rest of our task is to show the equivalence of reminders of equations (\ref{pl1}) and equations (\ref{pl2}) and $\tilde{\mathfrak{p}}^2=0$. In fact, we only need to consider the relation of three corresponding composited maps. In this subsection, we show that there is an one to one correspondence between composited maps in (\ref{ti1}) and (\ref{ti2}), and then we explore the relation of composited maps in (\ref{ti2}) and $\tilde{\mathfrak{p}}^2$. Based on these facts, we manage to give a proof of Theorem \ref{tpl}. Our proof is inspired from \cite{laj}.

To discuss the relation of composited maps in equations (\ref{pl1}) and equations (\ref{pl2}), we require the following lemma.
\begin{lemma}\label{38} With the same notations as Theorem \ref{tpl}, we have 
	\begin{enumerate}
		\item \begin{align*}
		&\hat{\mathfrak{p}}_i(\hat{\mathfrak{p}}_j(sx_1, \cdots,  sx_{j}),sx_{j+1},\cdots, sx_{i+j-1})\\
		=&\begin{cases}
		(-1)^{(j(i-1)+\sum\limits_{r=1}^{(i+j)/2-1}|x_{2r}|)}s\mathfrak{p}_i(\mathfrak{p}_j(x_1, \cdots, x_j),x_{j+1},\cdots,  x_{i+j-1}),&\mbox{$i+j$ even},\\
		-(-1)^{(j(i-1)+\sum\limits_{r=1}^{(i+j-1)/2}|x_{2r-1}|)}s\mathfrak{p}_i(\mathfrak{p}_j(x_1, \cdots, x_j),x_{j+1},\cdots,  x_{i+j-1}),&\mbox{$i+j$ odd}.
		\end{cases}
		\end{align*}
		\item \begin{align*}
		&\hat{\mathfrak{p}}_i(sx_1, \cdots, sx_{i-1}, \hat{\mathfrak{p}}_j(sx_{i}, \cdots, sx_{i+j-1}))\\
		=&\begin{cases}
		(-1)^{(\sum\limits_{r=1}^{(i+j)/2-1}|x_{2r}|+(j-1)(\sum\limits_{t=1}^{i-1}|x_{t}|))}s\mathfrak{p}_i(x_1,\cdots, x_{i-1}, \mathfrak{p}_j(x_{i}, \cdots, x_{i+j-1})),&\mbox{$i+j$ even},\\
		-(-1)^{(\sum\limits_{r=1}^{(i+j-1)/2}|x_{2r-1}|+(j-1)(\sum\limits_{t=1}^{i-1}|x_{t}|))}s\mathfrak{p}_i(x_1,\cdots, x_{i-1}, \mathfrak{p}_j(x_{i}, \cdots, x_{i+j-1})),&\mbox{$i+j$ odd}.
		\end{cases}
		\end{align*}
	\end{enumerate}
\end{lemma}
\begin{proof} The proof is is a simple calculation. Let $x_i\in V$ be homogeneous elements. Then 
\begin{align*}
&\hat{\mathfrak{p}}_i(\hat{\mathfrak{p}}_j(sx_1, \cdots,sx_{j}),sx_{j+1},\cdots, sx_{i+j-1})\\
=&\begin{cases}
-(-1)^{\sum\limits_{r=1}^{(j-1)/2}|x_{2r}|}\hat{\mathfrak{p}}_i(s\mathfrak{p}_j(x_1,\cdots,x_j),sx_{j+1},\cdots, sx_{i+j-1}),&\mbox{ if $j$ is odd,}\\
(-1)^{\sum\limits_{r=1}^{j/2}|x_{2r-1}|}\hat{\mathfrak{p}}_i(s\mathfrak{p}_j(x_1,\cdots,x_j),sx_{j+1},\cdots, sx_{i+j-1}),&\mbox{ if $j$ is even,}
\end{cases}\\
=&\begin{cases}
-(-1)^{\sum\limits_{r=1}^{(j-1)/2}|x_{2r}|}(-1)^{1+\sum\limits_{t=1}^{(i-1)/2}|x_{2r+j-1}|}\\s\mathfrak{p}_i(\mathfrak{p}_j(x_1, \cdots, x_j),x_{j+1},\cdots,  x_{i+j-1}),&\mbox{$i$ odd},\mbox{$j$ odd},\\
-(-1)^{\sum\limits_{r=1}^{(j-1)/2}|x_{2r}|}(-1)^{j-2+\sum\limits_{q=1}^{j}|x_q|+\sum\limits_{t=2}^{i/2}|x_{2t+j-2}|}\\s\mathfrak{p}_i(\mathfrak{p}_j(x_1, \cdots, x_j),x_{j+1},\cdots,  x_{i+j-1}),&\mbox{$i$ even},\mbox{$j$ odd},\\
(-1)^{\sum\limits_{r=1}^{j/2}|x_{2r-1}|}(-1)^{1+\sum\limits_{t=1}^{(i-1)/2}|x_{2r+j-1}|}\\s\mathfrak{p}_i(\mathfrak{p}_j(x_1, \cdots, x_j),x_{j+1},\cdots,  x_{i+j-1}),&\mbox{$i$ odd},\mbox{$j$ even},\\
(-1)^{\sum\limits_{r=1}^{j/2}|x_{2r-1}|}(-1)^{j-2+\sum\limits_{q=1}^{j}|x_q|+\sum\limits_{t=2}^{i/2}|x_{2t+j-2}|}\\s\mathfrak{p}_i(\mathfrak{p}_j(x_1, \cdots, x_j),x_{j+1},\cdots,  x_{i+j-1}),&\mbox{$i$ even},\mbox{$j$ even},
\end{cases}\\
=&\begin{cases}
(-1)^{(\sum\limits_{r=1}^{(i+j)/2-1}|x_{2r}|)}s\mathfrak{p}_i(\mathfrak{p}_j(x_1, \cdots, x_j),x_{j+1},\cdots,  x_{i+j-1}),&\mbox{$i$ odd},\mbox{$j$ odd},\\
(-1)^{(\sum\limits_{r=1}^{(i+j-1)/2}|x_{2r-1}|)}s\mathfrak{p}_i(\mathfrak{p}_j(x_1, \cdots, x_j),x_{j+1},\cdots,  x_{i+j-1}),&\mbox{$i$ even},\mbox{$j$ odd},\\
(-1)^{(1+\sum\limits_{r=1}^{(i+j-1)/2}|x_{2r-1}|)}s\mathfrak{p}_i(\mathfrak{p}_j(x_1, \cdots, x_j),x_{j+1},\cdots,  x_{i+j-1}),&\mbox{$i$ odd},\mbox{$j$ even},\\
(-1)^{(\sum\limits_{r=1}^{(i+j)/2-1}|x_{2r}|)}s\mathfrak{p}_i(\mathfrak{p}_j(x_1, \cdots, x_j),x_{j+1},\cdots,  x_{i+j-1}),&\mbox{$i$ even},\mbox{$j$ even},
\end{cases}\\
=&\begin{cases}
(-1)^{(j(i-1)+\sum\limits_{r=1}^{(i+j)/2-1}|x_{2r}|)}s\mathfrak{p}_i(\mathfrak{p}_j(x_1, \cdots, x_j),x_{j+1},\cdots,  x_{i+j-1}),&\mbox{$i+j$ even},\\
-(-1)^{(j(i-1)+\sum\limits_{r=1}^{(i+j-1)/2}|x_{2r-1}|)}s\mathfrak{p}_i(\mathfrak{p}_j(x_1, \cdots, x_j),x_{j+1},\cdots,  x_{i+j-1}),&\mbox{$i+j$ odd}.
\end{cases}
\end{align*}
By now we complete the proof of (1). Similarly, we can prove (2).
\end{proof}

Then we obtain the next proposition.
\begin{proposition}\label{p1}Let $(V,\{\mathfrak{p}_n\})$ be a $PL_{\infty}$-algebras in the sense of  Definition \ref{dal} and  $$\mathfrak{P}_n:=\sum\limits_{i+j=n+1}\sum\limits_{m=0}^{i-1}\frac{(-1)^{j(i-m-1)+m}}{(i-1)!(j-1)!}\mathfrak{p}_i\circ(I_{m}\otimes\mathfrak{p}_j\otimes I_{i-m-1})\circ(\rho^{(2)}_{w_{i+j-2}}\otimes I_1).$$  Then
	$\sum\limits_{i+j=n+1}\sum\limits_{m=0}^{i-1}\frac{1}{(i-1)!(j-1)!}\hat{\mathfrak{p}}_i\circ(I_{m}\otimes\hat{\mathfrak{p}}_j\otimes I_{i-m-1})\circ(\rho^{(1)}_{w_{i+j-2}}\otimes I_1)=-\hat{\mathfrak{P}}_n.$
\end{proposition}

\begin{proof}Assume that $i+j=n+1$ is even. According to Lemma \ref{per} and Lemma \ref{38}, one derives
	\begin{align*}
	&\sum\limits_{m=0}^{i-1}\frac{1}{(i-1)!(j-1)!}\hat{\mathfrak{p}}_i\circ(I_{m}\otimes\hat{\mathfrak{p}}_j\otimes I_{i-m-1})\circ(\rho^{(1)}_{w_{i+j-2}}\otimes I_1)(sx_1,\cdots,sx_n)\\
	=&\sum\limits_{\sigma \in Sh(j-1,1,i-2)}\epsilon(\sigma;sx_1,\cdots,sx_n)\hat{\mathfrak{p}}_i(\hat{\mathfrak{p}}_j(sx_{\sigma(1)},\cdots,sx_{\sigma(j)}),sx_{\sigma(j+1)},\cdots,sx_{\sigma{(i+j-2)}},sx_{i+j-1})\\
	&+\sum\limits_{\sigma \in Sh(i-1,j-1)}(-1)^{(\sum\limits_{r=1}^{i-1}|sx_{\sigma(r)}|)}\epsilon(\sigma;sx_1,\cdots,sx_n)\hat{\mathfrak{p}}_i(sx_{\sigma(1)},\cdots,sx_{\sigma(i-1)},\hat{\mathfrak{p}}_j(sx_{\sigma(i)},\cdots,\\
	&sx_{\sigma(i+j-2)},\cdots,sx_{i+j-1}))\\
	=&\sum\limits_{\sigma \in Sh(j-1,1,i-2)}\epsilon(\sigma;sx_1,\cdots,sx_n)(-1)^{(j(i-1)+\sum\limits_{r=1}^{(i+j)/2-1}|x_{\sigma(2r)}|)}s\mathfrak{p}_i(\mathfrak{p}_j(x_{\sigma(1)}, \cdots, x_{\sigma(j)}),x_{\sigma(j+1)}\\
	&,\cdots, x_{\sigma(i+j-2)}, x_{i+j-1})\\
	&+\sum\limits_{\sigma \in Sh(i-1,j-1)}(-1)^{(\sum\limits_{r=1}^{i-1}|sx_{\sigma(r)}|)}\epsilon(\sigma;sx_1,\cdots,sx_n)(-1)^{(\sum\limits_{r=1}^{(i+j)/2-1}|x_{\sigma(2r)}|+(j-1)(\sum\limits_{t=1}^{i-1}|x_{\sigma(t)}|))}\\
	&s\mathfrak{p}_i(x_{\sigma(1)},\cdots, x_{\sigma(i-1)}, \mathfrak{p}_j(x_{\sigma(i)}, \cdots, x_{\sigma(i+j-2)}, x_{i+j-1}))\\
	=&(-1)^{(\sum\limits_{r=1}^{(i+j)/2-1}|x_{2r}|)}\sum\limits_{\sigma \in Sh(j-1,1,i-2)}sgn(\sigma)\epsilon(\sigma)(-1)^{j(i-1)}s\mathfrak{p}_i(\mathfrak{p}_j(x_{\sigma(1)}, \cdots,  x_{\sigma(j)}),x_{\sigma(j+1)},\cdots,\\
	& x_{\sigma(i+j-2)}, x_{i+j-1})\\
	&+(-1)^{(\sum\limits_{r=1}^{(i+j)/2-1}|x_{2r}|)}\sum\limits_{\sigma \in Sh(i-1,j-1)}sgn(\sigma)\epsilon(\sigma)(-1)^{(i-1+j(\sum\limits_{t=1}^{i-1}|x_{\sigma(t)}|))}s\mathfrak{p}_i(x_{\sigma(1)},\cdots, x_{\sigma(i-1)}, \mathfrak{p}_j\\
	&(x_{\sigma(i)}, \cdots, x_{\sigma(i+j-2)}, x_{i+j-1}))\\
	=&(-1)^{(\sum\limits_{r=1}^{(i+j)/2-1}|x_{2r}|)}s(\sum\limits_{m=0}^{i-1}\frac{(-1)^{j(i-m-1)+m}}{(i-1)!(j-1)!}\mathfrak{p}_i\circ(I_{m}\otimes\mathfrak{p}_j\otimes I_{i-m-1})\circ(\rho^{(2)}_{w_{i+j-2}}\otimes I_1)\\
	&(x_1,\cdots,x_n)). 
	\end{align*}  Since $(-1)^{(\sum\limits_{r=1}^{(i+j)/2-1}|x_{2r}|)}=(-1)^{(\sum\limits_{r=1}^{(n+1)/2-1}|x_{2r}|)}$, it can be exchanged with $\sum\limits_{i+j=n+1}$. Thus we finish our proof when $n$ is odd. The other case follows from a similar calculation.
\end{proof}
We next consider the relation between the composited map of equations (\ref{pl2}) and the square of associated coderivation. 
\begin{proposition}\label{p2}
	Let $(V,\{\hat{\mathfrak{q}}_n\})$ be a $PL_{\infty}$-algebras in the sense of Definition \ref{dapl}. Then the component $P^kV\to P^{k-n+1}V$ of   $\tilde{\mathfrak{q}}^2$ is given by 
	\begin{align*}
	&\sum\limits_{\sigma \in Sh(n-1,1,k-n-1)}\epsilon(\sigma)\hat{\mathfrak{Q}}_{n}(x_{\sigma(1)},\cdots, x_{\sigma(n)})\wedge x_{\sigma(n+1)}\wedge\cdots\wedge x_{\sigma(k-1)}\otimes x_{k}\\
	&+\sum\limits_{\sigma \in Sh(k-n,n-1)}\epsilon(\sigma)x_{\sigma(1)}\wedge\cdots\wedge x_{\sigma(k-n)}\otimes\hat{\mathfrak{Q}}_{n}(x_{\sigma(k-n+1)},\cdots, x_{\sigma(k-1)}, x_{k}),
	\end{align*}
	where $$\hat{\mathfrak{Q}}_{n}:=\sum\limits_{i+j=n+1}\sum\limits_{m=0}^{i-1}\frac{1}{(i-1)!(j-1)!}\hat{\mathfrak{q}}_i\circ(I_{m}\otimes\hat{\mathfrak{q}}_j\otimes I_{i-m-1})\circ(\rho^{(1)}_{w_{i+j-2}}\otimes I_1).$$
\end{proposition}
\begin{proof} We first compute the component $P^kV\to P^{k-j+1}V\to P^{k-i-j+2}V$ of $\tilde{\mathfrak{q}}^2$.
	\begin{align*}
	&\tilde{\mathfrak{q}}^2(x_1\wedge\cdots\wedge x_{k-1}\otimes x_k)\\
	=&\sum\limits_{\sigma \in Sh(j-1,1,k-j-1)}\epsilon(\sigma)\tilde{\mathfrak{q}}(\hat{\mathfrak{q}}_{j}(x_{\sigma(1)},\cdots, x_{\sigma(j)})\wedge x_{\sigma(j+1)}\wedge\cdots\wedge x_{\sigma(k-1)}\otimes x_k)\\
	&+\sum\limits_{\sigma \in Sh(k-j,j-1)}(-1)^{\sum\limits_{t=1}^{k-j}|x_{\sigma(t)}|}\epsilon(\sigma)\tilde{\mathfrak{q}}(x_{\sigma(1)}\wedge\cdots\wedge x_{\sigma(k-j)}\otimes\hat{\mathfrak{q}}_{j}(x_{\sigma(k-j+1)},\cdots, x_{\sigma(k-1)}, x_k))\\
	=&\sum\limits_{\sigma \in Sh(j-1,1,i-2,1,k-i-j)}\epsilon(\sigma)\hat{\mathfrak{q}}_{i}(\hat{\mathfrak{q}}_{j}(x_{\sigma(1)},\cdots, x_{\sigma(j)}),x_{\sigma(j+1)},\cdots, x_{\sigma(i+j-1)})\wedge x_{\sigma(i+j)}\wedge\cdots\wedge\\
	& x_{\sigma(k-1)}\otimes x_k\\
	&+\sum\limits_{\sigma \in Sh(i-1,j-1,1,k-i-j)}(-1)^{\sum\limits_{r=1}^{i-1}|x_{\sigma(r)}|}\epsilon(\sigma)\hat{\mathfrak{q}}_{i}(x_{\sigma(1)},\cdots, x_{\sigma(i-1)},\hat{\mathfrak{q}}_{j}(x_{\sigma(i)},\cdots, x_{\sigma(i+j-1)}))\wedge\\
	& x_{\sigma(i+j)}\wedge\cdots\wedge x_{\sigma(k-1)}\otimes x_{k}\\
	&+\sum\limits_{\sigma \in Sh(i-1,1,j-1,1,k-i-j-1)}(-1)^{\sum\limits_{r=1}^{i}|x_{\sigma(r)}|}\epsilon(\sigma)\hat{\mathfrak{q}}_{i}(x_{\sigma(1)},\cdots, x_{\sigma(i)})\wedge \hat{\mathfrak{q}}_{j}(x_{\sigma(i+1)},\cdots, x_{\sigma(i+j)})\\
	&\wedge\cdots\wedge x_{\sigma(k-1)}\otimes x_{k}\\
	&+\sum\limits_{\sigma \in Sh(i-1,1,k-i-j,j-1)}(-1)^{\sum\limits_{t=1}^{k-j}|x_{\sigma(t)}|}\epsilon(\sigma)\hat{\mathfrak{q}}_{i}(x_{\sigma(1)},\cdots, x_{\sigma(i)})\wedge x_{\sigma(i+1)}\wedge\cdots\wedge x_{\sigma(k-j)}\otimes \hat{\mathfrak{q}}_{j}\\
	&(x_{\sigma(k-j+1)},\cdots, x_{\sigma(k-1)}, x_k)\\
	&+\sum\limits_{\sigma \in Sh(j-1,1,k-i-j,i-1)}(-1)^{1+\sum\limits_{r=1}^{k-i}|x_{\sigma(r)}|}\epsilon(\sigma)\hat{\mathfrak{q}}_{j}(x_{\sigma(1)},\cdots, x_{\sigma(j)})\wedge\cdots\wedge x_{\sigma(k-i)}\otimes\hat{\mathfrak{q}}_{i}(x_{\sigma(k-i+1)},\\
	&\cdots, x_{\sigma(k-1)}, x_{k})\\
	&+\sum\limits_{\sigma \in Sh(k-i-j+1,j-1,1,i-2)}\epsilon(\sigma)x_{\sigma(1)}\wedge\cdots\wedge x_{\sigma(k-i-j+1)}\otimes\hat{\mathfrak{q}}_{i}(\hat{\mathfrak{q}}_{j}(x_{\sigma(k-i-j+2)},\cdots, x_{\sigma(k-i+1)}),\\
	&\cdots, x_{\sigma(k-1)}, x_{k})\\
	&+\sum\limits_{\sigma \in Sh(k-i-j+1,i-1,j-1,1)}(-1)^{\sum\limits_{t=k-i-j+2}^{k-j}|x_{\sigma(t)}|}\epsilon(\sigma)x_{\sigma(1)}\wedge\cdots\wedge x_{\sigma(k-i-j+1)}\otimes\hat{\mathfrak{q}}_{i}(x_{\sigma(k-i-j+2)},\\
	&\cdots, x_{\sigma(k-j)}, \hat{\mathfrak{q}}_{j}(x_{\sigma(k-j+1)},\cdots, x_{\sigma(k-1)}, x_k)).
	\end{align*}
	
Note that component $P^kV\to P^{k-n+1}V$ of $\tilde{\mathfrak{q}}^2$ is the sum of components $P^kV\to P^{k-j+1}V\to P^{k-i-j+2}V$ satisfying $i+j=n+1$. So we add $\sum\limits_{i+j=n+1}$ to the above result, and the middle three items can be eliminated as follows.
\begin{align*}
&\sum\limits_{i+j=n+1}\sum\limits_{\sigma \in Sh(i-1,1,j-1,1,k-i-j-1)}(-1)^{\sum\limits_{r=1}^{i}|x_{\sigma(r)}|}\epsilon(\sigma)\hat{\mathfrak{q}}_{i}(x_{\sigma(1)},\cdots, x_{\sigma(i)})\wedge \hat{\mathfrak{q}}_{j}(x_{\sigma(i+1)},\cdots, x_{\sigma(i+j)})\\
&\wedge\cdots\wedge x_{\sigma(k-1)}\otimes x_{k}\\
=&\frac{1}{2}(\sum\limits_{i+j=n+1}\sum\limits_{\sigma \in Sh(i-1,1,j-1,1,k-i-j-1)}(-1)^{\sum\limits_{r=1}^{i}|x_{\sigma(r)}|}\epsilon(\sigma)\hat{\mathfrak{q}}_{i}(x_{\sigma(1)},\cdots, x_{\sigma(i)})\wedge \hat{\mathfrak{q}}_{j}(x_{\sigma(i+1)},\cdots,\\
& x_{\sigma(i+j)})\wedge\cdots\wedge x_{\sigma(k-1)}\otimes x_{k}\\
&+\sum\limits_{i+j=n+1}\sum\limits_{\sigma \in Sh(j-1,1,i-1,1,k-i-j-1)}(-1)^{\sum\limits_{r=1}^{j}|x_{\sigma(r)}|}\epsilon(\sigma)\hat{\mathfrak{q}}_{j}(x_{\sigma(1)},\cdots, x_{\sigma(j)})\wedge \hat{\mathfrak{q}}_{i}(x_{\sigma(j+1)},\cdots,\\
& x_{\sigma(i+j)})\wedge\cdots\wedge x_{\sigma(k-1)}\otimes x_{k})\\
=&\frac{1}{2}(\sum\limits_{i+j=n+1}\sum\limits_{\sigma \in Sh(i-1,1,j-1,1,k-i-j-1)}(-1)^{\sum\limits_{r=1}^{i}|x_{\sigma(r)}|}\epsilon(\sigma)\hat{\mathfrak{q}}_{i}(x_{\sigma(1)},\cdots, x_{\sigma(i)})\wedge \hat{\mathfrak{q}}_{j}(x_{\sigma(i+1)},\cdots,\\
& x_{\sigma(i+j)})\wedge\cdots\wedge x_{\sigma(k-1)}\otimes x_{k}\\
&+\sum\limits_{i+j=n+1}\sum\limits_{\sigma \in Sh(j-1,1,i-1,1,k-i-j-1)}(-1)^{1+\sum\limits_{r=1}^{i}|x_{\sigma(r)}|}\epsilon(\sigma)\hat{\mathfrak{q}}_{i}(x_{\sigma(1)},\cdots, x_{\sigma(i)})\wedge \hat{\mathfrak{q}}_{j}(x_{\sigma(i+1)},\cdots,\\
& x_{\sigma(i+j)})\wedge\cdots\wedge x_{\sigma(k-1)}\otimes x_{k})\\
=&0,
\end{align*}and 
\begin{align*}
&\sum\limits_{i+j=n+1}\sum\limits_{\sigma \in Sh(i-1,1,k-i-j,j-1)}(-1)^{\sum\limits_{t=1}^{k-j}|x_{\sigma(t)}|}\epsilon(\sigma)\hat{\mathfrak{q}}_{i}(x_{\sigma(1)},\cdots, x_{\sigma(i)})\wedge\cdots\wedge x_{\sigma(k-j)}\otimes \hat{\mathfrak{q}}_{j}\\
&(x_{\sigma(k-j+1)},\cdots, x_{\sigma(k-1)}, x_k)\\
&+\sum\limits_{i+j=n+1}\sum\limits_{\sigma \in Sh(j-1,1,k-i-j,i-1)}(-1)^{1+\sum\limits_{r=1}^{k-i}|x_{\sigma(r)}|}\epsilon(\sigma)\hat{\mathfrak{q}}_{j}(x_{\sigma(1)},\cdots, x_{\sigma(j)})\wedge\cdots\wedge x_{\sigma(k-i)}\otimes\hat{\mathfrak{q}}_{i}\\
&(x_{\sigma(k-i+1)},\cdots, x_{\sigma(k-1)}, x_{k})\\
=&\sum\limits_{i+j=n+1}\sum\limits_{\sigma \in Sh(i-1,1,k-i-j,j-1)}(-1)^{\sum\limits_{t=1}^{k-j}|x_{\sigma(t)}|}\epsilon(\sigma)\hat{\mathfrak{q}}_{i}(x_{\sigma(1)},\cdots, x_{\sigma(i)})\wedge\cdots\wedge x_{\sigma(k-j)}\otimes \hat{\mathfrak{q}}_{j}\\
&(x_{\sigma(k-j+1)},\cdots, x_{\sigma(k-1)}, x_k)\\
&+\sum\limits_{i+j=n+1}\sum\limits_{\sigma \in Sh(i-1,1,k-i-j,j-1)}(-1)^{1+\sum\limits_{t=1}^{k-j}|x_{\sigma(t)}|}\epsilon(\sigma)\hat{\mathfrak{q}}_{i}(x_{\sigma(1)},\cdots, x_{\sigma(i)})\wedge x_{\sigma(i+1)}\wedge\cdots\wedge x_{\sigma(k-j)}\\
&\otimes \hat{\mathfrak{q}}_{j}(x_{\sigma(k-j+1)},\cdots, x_{\sigma(k-1)}, x_k)\\
=&0.
\end{align*}
The remaining four terms can be simplified by $\hat{\mathfrak{Q}}_n$.
\begin{align*}
&\tilde{\mathfrak{q}}^2(x_1\wedge\cdots\wedge x_{k-1}\otimes x_k)\\
=&\sum\limits_{\sigma \in Sh(j-1,1,i-2,1,k-i-j)}\epsilon(\sigma)\hat{\mathfrak{q}}_{i}(\hat{\mathfrak{q}}_{j}(x_{\sigma(1)},\cdots, x_{\sigma(j)}),x_{\sigma(j+1)},\cdots, x_{\sigma(i+j-1)})\wedge x_{\sigma(i+j)}\wedge\cdots\wedge\\
& x_{\sigma(k-1)}\otimes x_k\\
&+\sum\limits_{\sigma \in Sh(i-1,j-1,1,k-i-j)}(-1)^{\sum\limits_{r=1}^{i-1}|x_{\sigma(r)}|}\epsilon(\sigma)\hat{\mathfrak{q}}_{i}(x_{\sigma(1)},\cdots, x_{\sigma(i-1)},\hat{\mathfrak{q}}_{j}(x_{\sigma(i)},\cdots, x_{\sigma(i+j-1)}))\wedge\\
& x_{\sigma(i+j)}\wedge\cdots\wedge x_{\sigma(k-1)}\otimes x_{k}\\
&+\sum\limits_{\sigma \in Sh(k-i-j+1,j-1,1,i-2)}\epsilon(\sigma)x_{\sigma(1)}\wedge\cdots\wedge x_{\sigma(k-i-j+1)}\otimes\hat{\mathfrak{q}}_{i}(\hat{\mathfrak{q}}_{j}(x_{\sigma(k-i-j+2)},\cdots, x_{\sigma(k-i+1)}),\\
&\cdots,x_{\sigma(k-1)}, x_{k})\\
&+\sum\limits_{\sigma \in Sh(k-i-j+1,i-1,j-1,1)}(-1)^{\sum\limits_{t=k-i-j+2}^{k-j}|x_{\sigma(t)}|}\epsilon(\sigma)x_{\sigma(1)}\wedge\cdots\wedge x_{\sigma(k-i-j+1)}\otimes\hat{\mathfrak{q}}_{i}(x_{\sigma(k-i-j+2)},\\
&\cdots, x_{\sigma(k-j)}, \hat{\mathfrak{q}}_{j}(x_{\sigma(k-j+1)}\wedge\cdots\wedge x_{\sigma(k-1)}\otimes x_k))\\
=&\sum\limits_{\sigma \in Sh(n-1,1,k-n-1)}\epsilon(\sigma)\hat{\mathfrak{Q}}_{n}(x_{\sigma(1)},\cdots, x_{\sigma(n)})\wedge x_{\sigma(n+1)}\wedge\cdots\wedge x_{\sigma(k-1)}\otimes x_{k}\\
&+\sum\limits_{\sigma \in Sh(k-n,n-1)}\epsilon(\sigma)x_{\sigma(1)}\wedge\cdots\wedge x_{\sigma(k-n)}\otimes\hat{\mathfrak{Q}}_{n}(x_{\sigma(k-n+1)},\cdots,x_{\sigma(k-1)}, x_{k}).\\
\end{align*}By now we have completed the proof of this proposition.
\end{proof}

Putting the above conclusions together, we can prove Theorem \ref{tpl} now.

\begin{proof}[Proof of Theorem \ref{tpl}]For a graded vector space $V$ equipped with a collection $\{\mathfrak{p}_n:\otimes^n V\to V,n\geq1\}$ of homogeneous linear maps of cohomological degree $n-2$, let $$\mathfrak{P}_n:=\sum\limits_{i+j=n+1}\sum\limits_{m=0}^{i-1}\frac{(-1)^{j(i-m-1)+m}}{(i-1)!(j-1)!}\mathfrak{p}_i\circ(I_{m}\otimes\mathfrak{p}_j\otimes I_{i-m-1})\circ(\rho^{(2)}_{w_{i+j-2}}\otimes I_1)$$
and $$\hat{\mathfrak{Q}}_{n}:=\sum\limits_{i+j=n+1}\sum\limits_{m=0}^{i-1}\frac{1}{(i-1)!(j-1)!}\hat{\mathfrak{p}}_i\circ(I_{m}\otimes\hat{\mathfrak{p}}_j\otimes I_{i-m-1})\circ(\rho^{(1)}_{w_{i+j-2}}\otimes I_1).$$
	By Proposition \ref{p1}, we have $ \hat{\mathfrak{Q}}_{n}=-\hat{\mathfrak{P}}_{n} $. Then \begin{align*}
	&\text{$(V,\{\mathfrak{p}_n\})$ satisfies Equation (\ref{pl1})},\\
	\Longleftrightarrow&\mathfrak{P}_n=0, \text{for all } n\geq 1,\\
	\Longleftrightarrow&\hat{\mathfrak{Q}}_{n}=0, \text{for all } n\geq 1,\\
\Longleftrightarrow&\text{$(sV,\{\hat{\mathfrak{p}}_n\})$ satisfies Equation (\ref{pl2})}.
	\end{align*}
	
	Analogously, we can obtain the equivalence of (\ref{ti2}) and (\ref{ti3}) by Proposition \ref{p2}.
	\begin{align*}
	&\text{$(sV,\{\hat{\mathfrak{p}}_n\})$ satisfies Equation (\ref{pl2})},\\
	\Longleftrightarrow&\hat{\mathfrak{Q}}_{n}=0, \text{for all } n\geq 1,\\
	\Longleftrightarrow&\tilde{\mathfrak{p}}^2=0,\\
	\end{align*}
	where the implication that  $\tilde{\mathfrak{p}}^2=0\Longrightarrow\hat{\mathfrak{Q}}_{n}=0,n\geq 1$ results from the fact that the component $P^nV\to V$ of $\tilde{\mathfrak{p}}^2$ is exactly $\hat{\mathfrak{Q}}_{n}$.
\end{proof}
\section{relation among $n$-ary algebras and   homotopy algebras}\label{s3}
In this section, we derive the relation of   homotopy algebras from results in previous section. In particular, we get the relation of $n$-ary algebras.

For simplicity, we use $(V,\{\mu_n\})$ to denote a  homotopy algebra in the sense of  Definition \ref{dal}, and $(V,\{\hat{\mu}_n\})$ to denote a   homotopy algebra in the sense of Definition \ref{dapl}. Maps between $(V,\{\mu_n\})$ are denoted by pure letters and maps between $(V,\{\hat{\mu}_n\})$ are denoted by letters with hats. If confusion will not occur, we simply denote the component of a map $f$ by $f$.

Using the functor $Hom(-,V)$ to  the commutative diagram in Lemma \ref{cd}, we derive
\begin{displaymath}
\xymatrix{Hom(T^*V,V)\ar[rr]^{\hat{\alpha}^*}\ar[dr]^{\hat{\gamma}^*}&&Hom(\wedge^*V,V)\\
	&Hom(P^*V,V)\ar[ur]^{\hat{\beta}^*}.&}
\end{displaymath}
This diagram gives commutators of $A_{\infty}$-algebras.
\begin{lemma}\cite{la}
	If $(V,\{\hat{\mathfrak{m}}_n\})$ is an $A_{\infty}$-algebra, then $(V,\{\hat{\alpha}(\hat{\mathfrak{m}}_n)\})$ is an $L_{\infty}$-algebra.
\end{lemma} 
The previous lemma can be naturally generalized to $PL_{\infty}$-algebras.
\begin{theorem}\label{tc}
	\begin{enumerate}
		\item Suppose that $(V,\{\hat{\mathfrak{m}}_n\})$ is an $A_{\infty}$-algebra. Then $(V,\{\hat{\gamma}(\hat{\mathfrak{m}}_n)\})$ is a $PL_{\infty}$-algebra.
		\item For a $PL_{\infty}$-algebra $(V,\{\hat{\mathfrak{p}}_n\})$, the collection $\{\hat{\beta}(\hat{\mathfrak{m}}_n)\}$ defines an $L_{\infty}$-algebra structure on $V$.
	\end{enumerate}
\end{theorem}
\begin{proof}
	Denote the associated coderivation of $\{\hat{\mathfrak{m}}_n\}$ resp. $\{\hat{\gamma}(\hat{\mathfrak{m}}_n)\}$ by $\tilde{\mathfrak{m}}$ resp. $\tilde{\mathfrak{q}}$. We first show that the following diagram is commutative.
	\begin{displaymath}
	\xymatrix{P^*V\ar[r]^{\hat{\gamma}}\ar[d]^{\tilde{\mathfrak{q}}}&T^*V\ar[d]^{\tilde{\mathfrak{m}}}\\
		P^*V\ar[r]^{\hat{\gamma}}&T^*V}
	\end{displaymath}
	Since $\hat{\gamma}=\hat{\alpha}\otimes I_1$ and $\pi\hat{\alpha}=Id_{\wedge^*V}$, where $\pi$ is defined in the proof of Lemma \ref{cd}, the equation $\hat{\gamma}\tilde{\mathfrak{q}}=\tilde{\mathfrak{m}}\hat{\gamma}$ holds if and only if the equation  $\tilde{\mathfrak{q}}=(\pi\otimes I_1)\tilde{\mathfrak{m}}\hat{\gamma}$ holds. Let us check their components $P^kV\to P^lV$ as follow. For any homogeneous elements $x_i\in V$, we have
	\begin{align*}
	&\tilde{\mathfrak{q}}(x_1\wedge\cdots\wedge x_{k-1}\otimes x_k)\\
	=&\sum\limits_{\sigma \in Sh(k-l,1,l-2)}\epsilon(\sigma)\hat{\gamma}(\hat{\mathfrak{m}}_{k-l+1})(x_{\sigma(1)},\cdots, x_{\sigma(k-l+1)})\wedge x_{\sigma(k-l+2)}\wedge\cdots\wedge x_{\sigma(k-1)}\otimes x_k\\
	&+\sum\limits_{\sigma \in Sh(l-1,k-l)}(-1)^{\sum\limits_{t=1}^{l-1}|x_{\sigma(t)}|}\epsilon(\sigma)x_{\sigma(1)}\wedge\cdots\wedge x_{\sigma(l-1)}\otimes\hat{\gamma}(\hat{\mathfrak{m}}_{k-l+1})(x_{\sigma(l)},\cdots, x_{\sigma(k-1)}, x_k)\\
	=&\sum\limits_{\sigma \in Sh(1,\cdots,1,l-2)}\epsilon(\sigma)\hat{\mathfrak{m}}_{k-l+1}(x_{\sigma(1)},\cdots, x_{\sigma(k-l+1)})\wedge x_{\sigma(k-l+2)}\wedge\cdots\wedge x_{\sigma(k-1)}\otimes x_k\\
	&+\sum\limits_{\sigma \in Sh(l-1,1,\cdots,1)}(-1)^{\sum\limits_{t=1}^{l-1}|x_{\sigma(t)}|}\epsilon(\sigma)x_{\sigma(1)}\wedge\cdots\wedge x_{\sigma(l-1)}\otimes\hat{\mathfrak{m}}_{k-l+1}(x_{\sigma(l)},\cdots, x_{\sigma(k-1)}, x_k)
	\end{align*}
	On the other hand, one obtains
	\begin{align*}
	&(\pi\otimes I_1)\tilde{\mathfrak{m}}\hat{\gamma}(x_1\wedge\cdots\wedge x_{k-1}\otimes x_k)\\
	=&\sum\limits_{\sigma \in \mathbb{S}_{k-1}}\epsilon(\sigma)(\pi\otimes I_1)\tilde{\mathfrak{m}}(x_{\sigma(1)}\otimes\cdots\otimes x_{\sigma(k-1)}\otimes x_k)\\
	=&\sum\limits_{\sigma \in \mathbb{S}_{k-1}}\sum\limits_{i=0}^{l-2}(-1)^{\sum\limits_{r=1}^{i}|x_{\sigma(r)}|}\epsilon(\sigma)(\pi\otimes I_1)(x_{\sigma(1)}\otimes\cdots\otimes x_{\sigma(i)}\otimes \hat{\mathfrak{m}}_{k-l+1}(x_{\sigma(i+1)},\cdots, x_{\sigma(i+k-l+1)})\\
	&\otimes \cdots\otimes x_{\sigma(k-1)}\otimes x_k)\\
	&+\sum\limits_{\sigma \in \mathbb{S}_{k-1}}(-1)^{\sum\limits_{r=1}^{l-1}|x_{\sigma(r)}|}\epsilon(\sigma)(\pi\otimes I_1)(x_{\sigma(1)}\otimes\cdots\otimes x_{\sigma(l-1)}\otimes \hat{\mathfrak{m}}_{k-l+1}(x_{\sigma(l)},\cdots, x_{\sigma(k-1)},x_k))\\
	=&\sum\limits_{\sigma \in \mathbb{S}_{k-1}}\sum\limits_{i=0}^{l-2}\frac{(-1)^{\sum\limits_{r=1}^{i}|x_{\sigma(r)}|}\epsilon(\sigma)}{(l-1)!}x_{\sigma(1)}\wedge\cdots\wedge x_{\sigma(i)}\wedge \hat{\mathfrak{m}}_{k-l+1}(x_{\sigma(i+1)},\cdots, x_{\sigma(i+k-l+1)})\wedge \cdots\wedge\\
	& x_{\sigma(k-1)}\otimes x_k\\
	&+\sum\limits_{\sigma \in \mathbb{S}_{k-1}}\frac{(-1)^{\sum\limits_{r=1}^{l-1}|x_{\sigma(r)}|}\epsilon(\sigma)}{(l-1)!}x_{\sigma(1)}\wedge\cdots\wedge x_{\sigma(l-1)}\otimes \hat{\mathfrak{m}}_{k-l+1}(x_{\sigma(l)},\cdots, x_{\sigma(k-1)},x_k)\\
	=&\sum\limits_{\sigma \in Sh(1,\cdots,1,l-2)}\epsilon(\sigma)\hat{\mathfrak{m}}_{k-l+1}(x_{\sigma(1)},\cdots, x_{\sigma(k-l+1)})\wedge x_{\sigma(k-l+2)}\wedge\cdots\wedge x_{\sigma(k-1)}\otimes x_k\\
	&+\sum\limits_{\sigma \in Sh(l-1,1,\cdots,1)}(-1)^{\sum\limits_{r=1}^{l-1}|x_{\sigma(r)}|}\epsilon(\sigma)x_{\sigma(1)}\wedge\cdots\wedge x_{\sigma(l-1)}\otimes\hat{\mathfrak{m}}_{k-l+1}(x_{\sigma(l)},\cdots, x_{\sigma(k-1)}, x_k)
	\end{align*}
	
	For an $A_{\infty}$-algebra $(V,\{\hat{\mathfrak{m}}_n\})$, we have $\tilde{\mathfrak{m}}^2=0$ by Lemma \ref{tm}. Then $\hat{\gamma}\tilde{\mathfrak{q}}^2=\tilde{\mathfrak{m}}^2\hat{\gamma}=0$. Because $\hat{\gamma}$ is injective by Lemma \ref{cd}, $\tilde{\mathfrak{q}}^2=0$. Applying Theorem \ref{tpl}, one derives $(V,\{\hat{\gamma}(\hat{\mathfrak{m}}_n)\})$ is a $PL_{\infty}$-algebra.
	
	 The other term can be deduced in a similar way.
\end{proof}
\begin{remark}
	The last result in Theorem \ref{tc} coincides with a result in \cite{lst} which is proved by constructing a graded Lie map. 
\end{remark}
As equivalent definitions,   homotopy algebras in the form of $(V,\{\mu_n\})$ have analogous relation. Replacing $\rho^{(1)}$ by $\rho^{(2)}$, we get the commutators of $(V,\{\mu_n\})$.
\begin{corollary}
	Define $\alpha:=\sum\limits_{n\geq1}\rho_{w_{n}}^{(2)}$, $\beta:=\sum\limits_{n\geq1}\sum\limits_{\sigma \in Sh(n-1,1)}\rho_{\sigma}^{(2)}$ and $\gamma:=\sum\limits_{n\geq1}\rho_{w_{n-1}}^{(2)}\otimes I_1$. 
	\begin{enumerate}
		\item Suppose that $(V,\{\mathfrak{m}_n\})$ is an $A_{\infty}$-algebra. Then $(V,\{\gamma(\mathfrak{m}_n)\})$ is a $PL_{\infty}$-algebra.
		\item For a $PL_{\infty}$-algebra $(V,\{\mathfrak{p}_n\})$, the collection $\{\beta(\mathfrak{m}_n)\}$ defines an $L_{\infty}$-algebra structure on $V$.
	\end{enumerate}
\end{corollary}
\begin{proof}
	Likewise, it is enough to verify the first claim. There is a commutative diagram. 
		\begin{displaymath}
	\xymatrix{\mathfrak{m}_n\ar[r]^{\gamma}\ar[d]^{\hat{}}&\mathfrak{q}_n\ar[d]^{\hat{}}\\
		\hat{\mathfrak{m}}_n\ar[r]^{\hat{\gamma}}\ar[u]^{}&\hat{\mathfrak{q}}_n\ar[u]^{}}
	\end{displaymath}
	Its commutativity can be deduced from $\widehat{\gamma(\mathfrak{m}_n)}=\hat{\gamma}(\hat{\mathfrak{m}}_n)$. Without loss of generality, we assume $n$ is even. 
	\begin{align*}
	&\widehat{\gamma(\mathfrak{m}_n)}(sx_1,\cdots,sx_n)\\
	=&(-1)^{\sum\limits_{i=1}^{n/2}|x_{2i-1}|}s\gamma(\mathfrak{m}_n)(x_1,\cdots,x_n)\\
	=&(-1)^{\sum\limits_{i=1}^{n/2}|x_{2i-1}|}\sum\limits_{\sigma \in \mathbb{S}_{n-1}}sgn(\sigma)\epsilon(\sigma)s\mathfrak{m}_n(x_{\sigma(1)},\cdots,x_{\sigma(n-1)},x_n)
	\end{align*}
	By Lemma \ref{Per}, we can derive the composited map in the other side as follow.
	 \begin{align*}
	 &(-1)^{\sum\limits_{i=1}^{n/2}|x_{2i-1}|}\sum\limits_{\sigma \in \mathbb{S}_{n-1}}sgn(\sigma)\epsilon(\sigma)s\mathfrak{m}_n(x_{\sigma(1)},\cdots,x_{\sigma(n-1)},x_n)\\
	 =&\sum\limits_{\sigma \in \mathbb{S}_{n-1}}(-1)^{\sum\limits_{i=1}^{n/2}|x_{\sigma(2i-1)}|}\epsilon(\sigma;sx_1,\cdots,sx_{n-1})s\mathfrak{m}_n(x_{\sigma(1)},\cdots,x_{\sigma(n-1)},x_n)\\
	 =&\sum\limits_{\sigma \in \mathbb{S}_{n-1}}\epsilon(\sigma;sx_1,\cdots,sx_{n-1})\hat{\mathfrak{m}}_n(sx_{\sigma(1)},\cdots,sx_{\sigma(n-1)},sx_n)\\
	 =&\hat{\gamma}(\hat{\mathfrak{m}}_n)(sx_1,\cdots,sx_n).
	 \end{align*}
Hence $\widehat{\gamma(\mathfrak{m}_n)}=\hat{\gamma}(\hat{\mathfrak{m}}_n)$.	 
	 If $(V,\{\mathfrak{m}_n\})$ is an $A_{\infty}$-algebra, then $(sV,\{\hat{\mathfrak{m}}_n\})$ is an $A_{\infty}$-algebra by Lemma \ref{tm}. Hence  $(sV,\{\hat{\gamma}(\hat{\mathfrak{m}}_n)\})$ is a $PL_{\infty}$-algebra by Theorem \ref{tc}. As a result of Theorem \ref{tpl}, $\{\gamma(\mathfrak{m}_n)\}$ gives a $PL_{\infty}$-algebra structure on $V$.
\end{proof}
Using Proposition \ref{ni}, we can get the relation of $n$-ary algebras immediately.
\begin{corollary}\label{co}
	Let $V$ be a vector space.
	\begin{enumerate}
		\item Every partially associative $n$-algebra $(V,\mathfrak{m})$ carries a pre Lie $n$-algebra structure $\mathfrak{p}$ defined by $$\mathfrak{p}(x_1,\cdots,x_n):=\sum\limits_{\sigma \in \mathbb{S}_{n-1}}sgn(\sigma)\mathfrak{m}(x_{\sigma(1)},\cdots,x_{\sigma(n-1)},x_n).$$
		\item Every pre Lie $n$-algebra $(V,\mathfrak{p})$ carries a Lie $n$-algebra structure $\mathfrak{l}$ defined by $$\mathfrak{l}(x_1,\cdots,x_n):=\sum\limits_{\sigma \in Sh(n-1,1)}sgn(\sigma)\mathfrak{p}(x_{\sigma(1)},\cdots,x_{\sigma(n)}).$$
	\end{enumerate}
\end{corollary}
Specially, we have
\begin{corollary}
	A pre Lie $n$-algebra $(V,\mathfrak{p})$ is a $n$-Lie admissible, that is, 
	$$\mathfrak{l}(x_1,\cdots,x_n):=\sum\limits_{\sigma \in \mathbb{S}_{n}}sgn(\sigma)\mathfrak{p}(x_{\sigma(1)},\cdots,x_{\sigma(n)})$$ is a Lie $n$-algebra structure.
\end{corollary}
\begin{proof}That is because
	\begin{align*}
	\sum\limits_{\sigma \in \mathbb{S}_{n}}sgn(\sigma)\mathfrak{p}(x_{\sigma(1)},\cdots,x_{\sigma(n)})=(n-1)!\sum\limits_{\sigma \in Sh(n-1,1)}sgn(\sigma)\mathfrak{p}(x_{\sigma(1)},\cdots,x_{\sigma(n)}),
	\end{align*}
	for a pre Lie $n$-algebra operation $\mathfrak{p}$.
\end{proof}

\end{document}